\documentclass[a4paper, 11pt]{amsart}
\usepackage{amsmath,amsthm}
\usepackage{amssymb}
\usepackage{amscd}
\usepackage{enumerate}

\usepackage[]{hyperref}
\hypersetup{
    colorlinks=true,       
    linkcolor=red,          
    citecolor=blue,        
    filecolor=magenta,      
    urlcolor=cyan           
}

\newtheorem{theorem}{Theorem}[section]

\newtheorem{corollary}[theorem]{Corollary}
\newtheorem{definition}[theorem]{Definition}

\newtheorem{lemma}[theorem]{Lemma}

\newtheorem{proposition}[theorem]{Proposition}
\newtheorem{remark}[theorem]{Remark}
\newtheorem{assumption}[theorem]{Assumption}
\DeclareMathOperator{\cnx}{div}

\DeclareMathOperator{\diff}{d}
\DeclareMathOperator{\pv}{pv}

\def\nc{\newcommand}
\def\be{\beta}

\def\lam{\lambda}

\def\ra{\rightarrow}
\def\la{\leftarrow}

\def\D{\la D\ra}

\nc\pa{\partial}

\nc\CC{\mathbb{C}}
\nc\RR{\mathbb{R}}
\nc\QQ{\mathbb{Q}}
\nc\ZZ{\mathbb{Z}}
\nc\NN{\mathbb{N}}

\def\ba{\begin{align}}
\def\bad{\begin{aligned}}
\def\be{\begin{equation}}
\def\ea{\end{align}}
\def\ead{\end{aligned}}
\def\ee{\end{equation}}
\def\e{\eqref}

\def\dalpha{\diff \! \alpha}

\def\deta{\diff \! \eta}

\def\dt{\diff \! t}

\def\dh{\diff \! h}
\def\dr{\diff \! r}

\def\dx{\diff \! x}
\def\dxi{\diff \! \xi}

\def\dy{\diff \! y}

\def\fract{\frac{\diff}{\dt}}
\def\fracr{\frac{\diff}{\dr}}

\def\defn{\mathrel{:=}}
\def\eps{\varepsilon}
\def\la{\left\vert}
\def\lA{\left\Vert}
\def\bla{\big\vert}
\def\blA{\big\Vert}
\def\le{\leq}
\def\les{\lesssim}
\def\mez{\frac{1}{2}}
\def\ra{\right\vert}
\def\rA{\right\Vert}
\def\bra{\big\vert}
\def\brA{\big\Vert}
\def\tdm{\frac{3}{2}}

\def\uq{\frac{1}{4}}

\def\xN{\mathbb{N}}
\def\xR{\mathbb{R}}

\begin{document}
\title{On the Cauchy problem for the Muskat equation 
with non-Lipschitz initial data}
\author{Thomas Alazard}
\thanks{E-mail address: thomas.alazard@ens-paris-saclay.fr, Universit{\'e} Paris-Saclay, ENS Paris-Saclay, CNRS,
	Centre Borelli UMR9010, avenue des Sciences,
	F-91190 Gif-sur-Yvette, Paris, France}
\author{Quoc-Hung Nguyen}
\thanks{E-mail address: qhnguyen@shanghaitech.edu.cn, ShanghaiTech University, 393 Middle Huaxia Road, Pudong,
			Shanghai, 201210, China.}
\date{}

\setlength{\baselineskip}{5mm}

\begin{abstract}
This article is devoted to the study of the Cauchy problem for the Muskat equation. We consider initial data belonging to the critical 
Sobolev space of functions with three-half derivative in~$L^2$, up to 
a fractional logarithmic correction. 
As a corollary, we obtain the first local and global well-posedness results for initial 
free surfaces which are not Lipschitz. 

\end{abstract}   

\maketitle
      
\vfill
            
\section{Introduction}
The Muskat equation is an important model in the analysis of free surface flows, which 
describes the dynamics of the interface separating two fluids whose velocities obey 
Darcy's law (\cite{darcy1856fontaines,Muskat}). 
Its two main features are that it is a fractional parabolic equation 
and a highly nonlinear equation. 
These two features are shared by several equations 
which have attracted a lot of attention in recent years, like the surface quasi-geostrophic equation, 
the Hele-Shaw equation or 
the fractional porous media equation, to name a few. 
Among these equations, a specificity of the Muskat equation is that it 
admits a beautiful compact formulation in terms of finite differences, as observed by C\'ordoba 
and Gancedo~\cite{CG-CMP}. 
The latter formulation allows to study the Cauchy problem by means of tools at the interface of 
harmonic analysis and nonlinear partial differential equations. In this direction, we are very much 
influenced by the recent works by Constantin, C{\'o}rdoba, Gancedo, Rodr{\'\i}guez-Piazza 
and Strain~\cite{CCGRPS-JEMS2013,CCGRPS-AJM2016}, C\'ordoba and Lazar~\cite{Cordoba-Lazar-H3/2} 
and Gancedo and Lazar~\cite{Gancedo-Lazar-H2}. 

Our goal is to introduce for the Muskat problem an approach based on 
a logarithmic correction to the usual Gagliardo semi-norms which is adapted to both 
the fractional and nonlinear features of the equation, following earlier works 
in~\cite{Alazard-Lazar,BN18a,BN18b,BN18c,BN18d,Ng}. 
Our main result is stated after 
we introduce some notations, but one can express its main corollary as follows: 
one can study the Cauchy problem in an almost critical Sobolev space, 
allowing initial data which are not Lipschitz. 

\clearpage

\subsection{The Muskat equation}
Consider the dynamics of a time-dependent 
curve $\Sigma(t)$ separating two $2D$-domains $\Omega_1(t)$ and $\Omega_2(t)$. 
On the supposition that $\Sigma(t)$ is the graph of some function, we introduce the following notations
\begin{align*}
\Omega_1(t)&=\left\{ (x,y)\in \xR\times \xR\,;\, y>f(t,x)\right\},\\
\Omega_2(t)&=\left\{ (x,y)\in \xR\times \xR\,;\, y<f(t,x)\right\},\\
\Sigma(t)&=\left\{ (x,y)\in \xR\times \xR\,;\, y=f(t,x)\right\}.
\end{align*}
Assume 
that each domain~$\Omega_j$, $j=1,2$, 
is occupied by an incompressible fluid with constant density~$\rho_j$ and 
denote by $\rho=\rho(t,x)$ the function with value 
$\rho_j$ for $x\in\Omega_j(t)$. We assume that $\rho_2>\rho_1$ so that the heavier fluid is underneath the lighter one. 
Then the motion is determined by the incompressible porous media equations, where the velocity 
field $v$ is given by Darcy's law:
\be\label{MuskatrhovP}
\left\{
\begin{aligned}
&\partial_t\rho+\cnx(\rho v)=0,\\
&\cnx v=0,\\
&v+\nabla (P+\rho gy)=0,
\end{aligned}
\right.
\ee
where $g$ is the acceleration of gravity. 

Changes of unknowns, reducing the problem~\e{MuskatrhovP} to an evolution 
equation for the free surface parametrization, 
have been known for quite a time (see~\cite{CaOrSi-SIAM90,EsSi-ADE97,PrSi-book,SCH2004}). 
This approach was further developed by C\'ordoba and Gancedo~\cite{CG-CMP}
who obtained a beautiful compact formulation of the Muskat equation. Indeed, 
they showed that the Muskat problem is equivalent to the following 
equation for the free surface elevation:
\begin{align}\label{eq2.1}
\partial_tf=\frac{\rho}{2\pi}\pv \int_\xR\frac{\partial_x\Delta_\alpha f}{1+\left(\Delta_\alpha f\right)^2}\dalpha,
\end{align}
where 
the integral is understood in the sense of principal values, 
$\rho=\rho_2-\rho_1$ is the difference of the densities 
of the two fluids and $\Delta_\alpha f$ is the slope
\begin{align}\label{eq2.2}
\Delta_\alpha f(t,x)=\frac{f(t,x)-f(t,x-\alpha)}{\alpha}\cdot
\end{align}
Since $\rho_2>\rho_1$ by assumption, we may set $\rho=2$ without loss of generality. 

A key feature of this problem is 
that~\e{eq2.1} is preserved by the change of unknowns:
$$
f(t,x)\mapsto \frac{1}{\lambda}f\left(\lambda t,\lambda x\right).
$$
Hence, the two natural critical spaces for the initial data are the homogeneous spaces
$$
\dot{H}^{\frac{3}{2}}(\xR),\quad \dot{W}^{1,\infty}(\xR).
$$
The analysis of the Cauchy problem for the Muskat equation is now well developed, 
including global existence results under mild smallness assumptions and blow-up 
results for some large enough initial data. 
Local well-posedness results go back to the works of Yi~\cite{Yi2003}, 
Ambrose~\cite{Ambrose-2004,Ambrose-2007}, 
C\'ordoba and Gancedo~\cite{CG-CMP}, C\'ordoba, C\'ordoba and Gancedo~\cite{CCG-Annals}, 
Cheng, Granero-Belinch\'on, 
Shkoller~\cite{Cheng-Belinchon-Shkoller-AdvMath}. 
Then local well-posedness results were obtained in 
the sub-critical spaces by 
Constantin, Gancedo, Shvydkoy and Vicol~\cite{CGSV-AIHP2017} 
for initial data in the Sobolev space 
$W^{2,p}(\xR)$ for some $p>1$, and Matioc~\cite{Matioc1,Matioc2} for initial data 
in $H^s(\xR)$ with $s > 3/2$ (see also~\cite{Alazard-Lazar,Nguyen-Pausader}). 
Since the Muskat equation is parabolic, the proof of the local well-posedness 
results also gives global well-posedness results under a smallness assumption, see Yi~\cite{Yi2003}. 
The first global well-posedness results 
under mild smallness assumptions, namely assuming 
that the Lipschitz semi-norm is smaller than $1$, was 
obtained by Constantin, C{\'o}rdoba, Gancedo, Rodr{\'\i}guez-Piazza 
and Strain~\cite{CCGRPS-AJM2016} (see also \cite{CGSV-AIHP2017,PSt}). 

On the other hand, there are blow-up results for some large enough data 
by Castro, C\'{o}rdoba, Fefferman, Gancedo and 
L\'opez-Fern\'andez~(\cite{CCFG-ARMA-2013,CCFG-ARMA-2016,CCFGLF-Annals-2012}). 
They prove the existence of solutions such that 
at time $t=0$ the interface is a graph, at a later time $t_1>0$ 
the interface is no longer a graph and then at a subsequent time $t_2>t_1$, 
the interface is $C^3$ but not $C^4$. 

The previous discussion raises a question about the possible existence of a 
criteria on the slopes of the solutions which 
would force/prevent them to enter the unstable regime where the slope is infinite. 
Surprisingly, it is possible to solve the Cauchy problem for initial 
data whose slope can be arbitrarily large. 
Deng, Lei and Lin in~\cite{DLL} obtained the first result in this direction, under the assumption that the 
initial data are monotone. Cameron \cite{Cameron} proved 
the existence of a modulus of continuity for the derivative, and hence 
a global existence result assuming only that 
the product of the maximal and minimal slopes is bounded by $1$; thereby 
allowing arbitrarily large slopes too 
(recently, Abedin and Schwab also obtained the existence of a modulus of continuity 
in~\cite{Abedin-Schwab-2020} via Krylov-Safonov estimates). 
Then, by using a new formulation of the Muskat equation involving oscillatory integrals, 
C\'ordoba and Lazar established 
in \cite{Cordoba-Lazar-H3/2} that 
the Muskat equation is globally well-posed in time, assuming only that 
the initial data is sufficiently smooth and that the $\dot H^{3/2}(\xR)$-norm is small enough. 
This result was extended to the 3D case by Gancedo and Lazar~\cite{Gancedo-Lazar-H2}. 
Let us also quote papers by Vazquez~\cite{Vazquez-DCDS}, 
Granero-Belinch{\'o}n and Scrobogna~\cite{Granero-Scrobogna} for related global existence 
results for different equations. The existence and possible non-uniqueness 
of weak-solutions has also been thoroughly studied (we refer the reader to~\cite{Brenier2009,cordoba2011lack,szekelyhidi2012relaxation,castro2016mixing,forster2018piecewise,noisette2020mixing}).

\subsection{Fractional logarithmic spaces}
Based on the discussion earlier, one of the main questions left open is to solve the 
Cauchy problem for the Muskat equation 
for initial data which are not Lipschitz. Indeed, for such data, 
the slope is not only arbitrarily large but can be infinite. 
To prove the existence of such solutions, 
the main difficulties one has to cope with are the following: 
Firstly,  there is a degeneracy in the parabolic behavior when $f_x$ is not controlled 
(this is easily seen by looking at the energy estimate~\e{i5}~below: 
when $f_x$ is not controlled, one does not control the $L^2_{t,x}$-norm of the derivatives). Secondly, in addition to this degeneracy, one cannot apply classical nonlinear estimates. 
Indeed the latter 
require to control the $L^\infty$-norm of 
some factors, which amounts here to control the $L^\infty$-norm of the slopes~$\Delta_\alpha f$, equivalent to control 
the Lipschitz norm of~$f$. 
To overcome these difficulties, we will use two different kind of arguments, following earlier works in~\cite{Alazard-Lazar,BN18a,BN18b}. Firstly, we will prove estimates valid in 
critical spaces, by exploiting various cancellations 
as well as specific inequalities. 
Secondly, we will perform energy estimates 
in some variants of the classical Sobolev spaces, 
allowing to control a {\em fraction} of a logarithmic derivative. 
More precisely, the idea followed in this paper is to estimate the 
following norms.

\begin{definition}\label{defi:1}Given $a\ge 0$ and $s\ge 0$, the fractional logarithmic space 
$\mathcal{H}^{s,a}(\xR)$ consists of those functions 
$g\in L^2(\xR)$ such that the following norm is finite:
$$
\lA g\rA_{\mathcal{H}^{s,a}}^2=\int_{\xR} \left(1+|\xi|^2\right)^s\left( \log(4+|\xi|)\right)^{2a}\la \hat{g}(\xi)\ra^2\dxi.
$$
\end{definition}
\begin{remark}
$(i)$ 
Since the formulation of the Muskat equation involves 
the finite differences of $f$, it is important to notice 
that these semi-norms can be defined in 
terms of finite differences. We will see that, if $s\in (0,2)$, 
$$
\lA g\rA_{\mathcal{H}^{s,a}}^2\sim \lA g\rA_{L^2(\xR)}^2+
\iint_{\xR^2} \frac{\la 2g(x)-g(x+h)-g(x-h)\ra^2}{|h|^{2s}} \left[\log\left(4+\frac{1}{|h|^2}\right)\right]^{2a}\frac{\dx\dh}{|h|},
$$
and if $s=0$
\begin{align*}
\lA g\rA_{\mathcal{H}^{0,a}}^2&\sim \lA g\rA_{L^2(\xR)}^2\\&+
\iint_{\xR^2} \mathbf{1}_{|h|<\frac{1}{2}} \la 2g(x)-g(x+h)-g(x-h)\ra^2 
\left[\log\left(4+\frac{1}{|h|^2}\right)\right]^{-1+2a}\frac{\dx\dh}{|h|}\cdot
\end{align*}
The latter norms
were introduced in~\cite{BN18a} for $s\in [0,1)$ (with the symmetric 
difference replaced by $g(x+h)-g(x)$). 

$(ii)$ Here the word `fractional' is used to insist on the fact that $a$ 
belongs to $(0,1]$. This is important in view of~\e{n40} below. 
\end{remark}

We consider initial data in $\mathcal{H}^{3/2,a}(\xR)$ for some $a\ge 0$. 
Notice that the latter spaces lie 
between the Sobolev spaces $H^{3/2}(\xR)$ and $H^{3/2+\epsilon}(\xR)$:
$$
\forall \eps>0,~\forall a\ge 0,\quad 
H^{\tdm+\epsilon}(\xR)\subset \mathcal{H}^{\tdm,a}(\xR)\subset\mathcal{H}^{\tdm,0}(\xR)=H^{\tdm}(\xR).
$$
The definition of the Sobolev spaces is recalled below in \e{defi:Sobolev}. 
For our purposes, the most important think to note is that
\be\label{n40}
\mathcal{H}^{\tdm,a}(\xR)\subset W^{1,\infty}(\xR) 
\quad \text{if and only if}\quad a>\mez\cdot
\ee
It follows from~\e{n40} that there is a key dichotomy between the cases $a\le 1/2$ and $a>1/2$. 
Loosely speaking, for $a>1/2$, the analysis of the Cauchy problem in $\mathcal{H}^{3/2,a}(\xR)$ 
is expected to be similar to the one in sub-critical spaces $H^s(\xR)$ with $s>3/2$. 
While for $a\le 1/2$, the same problem is expected to be much more 
involved since one cannot control the $W^{1,\infty}$-norm of $f$ (which is  
ubiquitous in the estimates of nonlinear quantities involving gradients 
or the slopes $\Delta_\alpha f$).

\subsection{Main results} 
Once the fractional logarithmic spaces have been introduced, 
the question of solving the Cauchy problem for non Lipschitz initial data can be made precise. 
Namely, our goal is to prove that the Cauchy problem is well-posed on $\mathcal{H}^{3/2,a}(\xR)$ for 
some $a\le \frac{1}{2}$. Our main results assert that in fact one can solve the Cauchy problem 
down to $a= \frac{1}{3}$.

Recall that we set $\rho_2-\rho_1=2$, so that the Muskat equation~\e{eq2.1} reads
\begin{align}\label{eq2.1b}
\partial_tf=\frac{1}{\pi}\pv \int_\xR\frac{\partial_x\Delta_\alpha f}{1+\left(\Delta_\alpha f\right)^2}\dalpha.
\end{align}

\begin{theorem}[local well-posedness]\label{theo:main}
For any 
initial data $f_0$ in $\mathcal{H}^{\frac{3}{2},\frac{1}{3}}(\xR)$, 
there exists a positive time $T$ such that the Cauchy problem for the Muskat equation~\e{eq2.1b} has 
a unique solution
$$
f\in C^0\big([0,T];\mathcal{H}^{\tdm,\frac{1}{3}}(\xR)\big)
\cap L^2\big(0,T;H^{2}(\xR)\big).
$$ 
\end{theorem}
\begin{remark}\label{R:1.4}
The case $a=\frac{1}{3}$ corresponds to a limiting case. 
In particular, the time of existence does not depend only on the norm of $f_0$ but also on $f_0$ itself. 
More precisely, we will estimate the solution for a norm whose definition depends on~$f_0$  
(this can be understood by looking at Lemma~$\ref{L:critical}$ and Remark~$\ref{R:3.8}$). 
\end{remark}
We now give a global in time well-posedness result under a smallness condition on the following quantity:
\be\label{n141}
\begin{aligned}
\lA f_0\rA_{\tdm,\frac{1}{3}}^2&\defn \int_{\xR} |\xi|^{3} \log(4+|\xi|)^{\frac{2}{3}}\bla \hat{f_0}(\xi)\bra^2\dxi\\
&\sim\iint_{\xR^2} \frac{\la 2f_0(x)-f_0(x+h)-f_0(x-h)\ra^2}{|h|^{3}}
\left(\log\left(4+\frac{1}{|h|^2}\right)\right)^{\frac{2}{3}}\frac{\dx\dh}{|h|}\cdot
\end{aligned}
\ee
\begin{theorem}[global well-posedness]\label{theo:main2}
There exists a positive constant $c_0$ such that, for all 
initial data $f_0$ in $\mathcal{H}^{3/2,1/3}(\xR)$ satisfying 
\be\label{i10}
\lA f_0\rA_{\tdm,\frac13}\left( \lA f_0\rA_{L^2}^2+1\right)  \leq  c_0,
\ee
the Cauchy problem for the Muskat equation~\e{eq2.1b} has a unique solution
$$
f\in C^0\big([0,+\infty);\mathcal{H}^{\tdm,\frac{1}{3}}(\xR)\big)
\cap L^2\big(0,+\infty;H^{2}(\xR)\big).
$$ 
\end{theorem}

\subsection{Strategy of the proof and plan of the paper}
To prove Theorem~\ref{theo:main}, the key point is 
to work with critical-type norms. 
In this direction, we will prove some technical estimates which we think are of independent interest. 
To state the main consequence of the latter, let us 
introduce a bit of notation. We denote by $\D^{s,\phi}$ the Fourier multiplier $(-\Delta)^{s/2}\phi(D_x)$. 
Then, our main technical estimate asserts that, for any $\phi$, there holds
\be\label{i5}
\fract \blA \D^{\tdm,\phi}f\brA_{L^2}^2
+ \int_\xR \frac{\bla\D^{2,\phi}f\bra^2}{1+(\partial_x f)^2}\dx\le C Q(f) \blA\D^{2,\phi}f\brA_{L^2},
\ee
where
\begin{align*}
Q(f)&= \left(\lA f\rA_{\dot H^2}+\lA f\rA_{\dot H^{\frac{7}{4}}}^2\right) 
\blA\D^{\tdm,\phi}f\brA_{L^2}
+\blA\D^{\frac74,\phi}f\brA_{L^2}
\lA f\rA_{H^{\frac74}}\\
&\quad+\left(\lA f\rA_{H^{\frac{19}{12}}}^{3/2}+\lA f\rA_{\dot H^{\frac74}}^{1/2}\right)
\blA\D^{\frac{7}{4},\phi^{2}}f\brA^{1/2}_{L^2}
\lA f\rA_{\dot H^{\frac74}}.
\end{align*}
The crucial point is that the quantity $Q(f)$ does not involve the $W^{1,\infty}$-norm of~$f$.

To prove~\e{i5}, notable technical aspects include 
the proof of new commutator estimates and 
the systematic use of Triebel-Lizorkin norms. We develop these tools in~$\S\ref{S:2}$. 
With these results in hands, we begin in~$\S\ref{S:3}$ by introducing 
a sequence of approximate equations by a Galerkin type decomposition, which admit 
approximate solutions $(f_n)_{n\in \xN}$. 
Then we prove that the estimate \e{i5} holds for these approximate systems, uniformly in~$n$. 

We then conclude the proof of Theorem~\ref{theo:main} in two steps, 
by applying~\e{i5} with some special choice for $\phi$, satisfying $\phi(\xi)\sim (\log(4+\la\xi\ra))^{a}$. As already mentioned, one of the main difficulty is that the 
factor $1+(\partial_x f)^2$ which appears in the left hand-side of~\e{i5} is not controlled in $L^\infty_{t,x}$. To overcome this difficulty, 
we prove some new interpolation inequalities to estimate 
the factor $1+(\partial_x f)^2$, using the fractional logarithmic norms, by some quantity which is not bounded in time. To be more specific, assume that 
$\phi(\xi)\sim (\log(4+\la\xi\ra))^{a}$, 
and introduce the quantities
\begin{align*}
A(t)&=\blA \D^{\tdm,\phi}f(t)\brA_{L^2}^2,\\
B(t)&=\blA \D^{2,\phi}f(t)\brA_{L^2}^2.
\end{align*}
In \S\ref{S:3.2}, we will prove 
after a fair amount of bookkeeping an estimate of the form
$$
\fract A(t)+C_1\delta(t)B(t)\leq C_2 \left(\log\left(\frac{B(t)}{A(t)}\right)\right)^{-a} \left( \sqrt{A(t)}+A(t) \right)B(t),
$$
where
$$
\delta(t)\sim\left(1+ \log\left(4+\frac{B(t)}{A(t)+ \Vert f_0\Vert_{L^2}^2}\right)^{1-2a}\left( A(t)+ \Vert f_0\Vert_{L^2}^2\right)\right)^{-1}.
$$
Notice that $\delta(t)$ is not bounded from below so that the left-hand side is insufficient to control $B(t)$. However, to apply a Gronwall type inequality 
it will suffice to have $a> 1-2a$, that is $	a> \frac{1}{3}$, see~\ref{S:3.3}. The limiting case $a=\frac{1}{3}$ will be studied in~\S\ref{S:critical} 
by introducing a more general weight $\phi$ which is not 
a fraction of a logarithm, and whose definition depends on the initial data itself. 
This gives uniform bounds 
in $L^\infty_t(\mathcal{H}^{s,a}(\xR))$ for any $a\ge \frac{1}{3}$, 
for the approximate solutions $f_n$, 
from which we deduce the existence of a solution to the Muskat equation by extracting a subsequence. The uniqueness is proved in~\S\ref{S:3.5} by similar arguments, using again some delicate interpolation inequalities to handle the lack of Lipschitz control.

\subsection*{Notations}
Most notations are introduced in the next section. 
In particular, the definitions of Sobolev, Besov and Triebel-Lizorkin spaces are recalled in~\S\ref{S:2.1}.
To avoid possible confusions in the notations, we mention that, throughout the paper:
\begin{itemize}
\item We will sometimes 
write $\log (4+|\xi|)^a$ as a 
short notation for $(\log(4+|\xi|))^a$. 
\item All functions are assumed to be real-valued in this paper. 
Nevertheless, we will often use the complex-modulus notation in writing $\la \Delta_\alpha f\ra^2$ or $\la \alpha\ra^2$ in many identities, since we think 
it might help the reader to read the latter.
\item Given~$0\le t\le T$, a normed space~$X$
and a function~$\varphi=\varphi(t,x)$ defined on~$[0,T]\times \xR$ with values in~$X$, 
we denote by~$\varphi(t)$ the function~$x\mapsto
\varphi(t,x)$, and $\lA \varphi\rA_X$ is a short notation for the time dependent function 
$t\mapsto \lA \varphi(t)\rA_X$. 
\end{itemize}
\section{Nonlinearity and fractional derivatives in the Muskat problem}\label{S:2}

We now develop the linear and nonlinear tools needed to study the Muskat problem 
in the spaces $\mathcal{H}^{3/2,a}(\xR)$. 
The first paragraph is a review consisting of various notations and usual results about 
Besov and Triebel-Lizorkin spaces, 
which serve as the requested background for what follows. 
Then we study in $\S\ref{S:2.2}$ Fourier multipliers of the form 
$\D^s \phi(\la D_x\ra)$ for some symbols $\phi(|\xi|)$ 
which generalize the fractional logarithm $(\log(4+\la \xi\ra))^a$ introduced in the introduction. 
In particular we give a characterization of the space 
$$
\mathcal{H}^{s,\phi}(\xR)=\{ f\in L^2(\xR)\, :\, \D^s \phi(\la D_x\ra)f\in L^2(\xR)\},
$$
in terms of modified Gagliardo semi-norms. 
Then in \S\ref{S:2.3} we recall the paralinerization formula 
for the Muskat equation from~\cite{Alazard-Lazar}. 
The core of this section is \S\ref{S:2.4}, in which we prove technical ingredients 
needed to estimate the coefficients of the latter paralinearization formula in terms of the $\mathcal{H}^{3/2,\phi}(\xR)$-norms.

\subsection{Triebel-Lizorkin norms}\label{S:2.1}
This work builds on the analysis of the Muskat equation by C{\'o}rdoba and Lazar~\cite{Cordoba-Lazar-H3/2} and Alazard and Lazar~\cite{Alazard-Lazar}, which introduced the use of techniques related to Besov spaces in this problem (see 
also Gancedo and Lazar~\cite{Gancedo-Lazar-H2}). 
Here we will also use Triebel-Lizorkin spaces. 
For ease of reading, we recall various notations and results about these spaces,  
which will be used continually in the rest of the paper. 

Given a function $f\colon\xR\to\xR$, an integer $m\in\xN\setminus\{0\}$ 
and a real number $h\in\xR$,
 we define the finite differences 
$\delta_h^mf$ as follows:
$$
\delta_hf(x)=f(x)-f(x-h),\qquad \delta_h^{m+1}f=\delta_h(\delta_h^mf).
$$

\begin{definition}
Consider an integer $m\in\xN\setminus\{0\}$, a real number $s\in [m-1,m)$ and 
two real numbers $(p,q)$ in $[1,\infty)^2$. 
The homogeneous Triebel-Lizorkin space $\dot F^{s}_{p,q}(\xR)$ 
consists of those tempered distributions $f$
whose Fourier transform is integrable near the origin and such that 
\be\label{n5}
\lA f\rA_{\dot F^s_{p,q}}
=\left(\int_{\mathbb{R}}\left(\int_{\mathbb{R}} \la\delta_h^mf(x)\ra^q \frac{\dh}{|h|^{1+qs}}\right)^{\frac{p}{q}}\dx\right)^{\frac{1}{p}}<+\infty.
\ee
\end{definition}
\begin{remark}$i)$ 
We refer to Triebel~\cite[$\S 2.3.5$]{Triebel-TFS} for historical comments 
about these spaces, and Triebel~\cite[section~$3$]{Triebel1988} 
for the equivalence between this definition and 
other ones including the Littlewood-Paley decomposition. 

$ii)$ For the sake of comparison, recall that the Besov space $\dot B^{s}_{p,q}(\xR)$ consists of those 
tempered distributions $f$
whose Fourier transform is integrable near the origin and such that
\be\label{defi:B-spaces}
\lA f \rA_{\dot B^{s}_{p,q}} =
\left(\int_{\mathbb{R}}\left( \int_{\xR}\la\delta_h^mf(x)\ra^p\dx\right)^{\frac{q}{p}}
\frac{\dh}{|h|^{1+qs}}\right)^{\frac{1}{q}}<+\infty.
\ee
Notice that Besov defined his spaces 
in this way, that is with finite differences, see~\cite{Besov}.
\end{remark}
For easy reference, we recall two results allowing to compare the 
Triebel-Lizorkin semi-norms to the homogeneous and non-homogeneous Sobolev norms, 
which are defined by
\begin{equation}\label{defi:Sobolev}
\left\lVert u \right\rVert_{\dot{H}^{\sigma}}^{2} \defn (2\pi)^{-1} 
\int_{\xR} \la\xi\ra^{2\sigma} |\hat{u}(\xi)|^2\dxi,\quad 
\left\lVert u \right\rVert_{H^\sigma}^{2} \defn (2\pi)^{-1} 
\int_{\xR} (1+|\xi|^2)^{\sigma} |\hat{u}(\xi)|^2\dxi,
\end{equation}
where~$\widehat{u}$ is the Fourier transform of~$u$.

Recall that $\lA \cdot\rA_{\dot H^{s}}$ and $\lA \cdot\rA_{\dot{F}^s_{2,2}}$ 
are equivalent. 
Moreover, for $s\in (0,1)$,
\be\label{SE0}
\lA u\rA_{\dot H^{s}}^2=\frac{1}{4\pi c(s)}\lA u\rA_{\dot{F}^s_{2,2}}^2
\quad\text{with}\quad c(s)=\int_\xR \frac{1-\cos(h)}{\la h\ra^{1+2s}}\dh.
\ee

We will also extensively use the following Sobolev embeddings: 
For any $2<p_1< \infty$ and any $q\leq\infty$, if $s-\frac{1}{2}=s_1-\frac{1}{p_1}$ then 
\begin{equation}\label{FB}
\lA f\rA_{\dot F^{s_1}_{p_1,q}(\mathbb{R})}\leq C \lA f\rA_{\dot H^{s}(\mathbb{R})}.
\end{equation}

\subsection{Some special Fourier multipliers}\label{S:2.2}

Let us introduce a bit of notation which will be used 
continually in the rest of the paper. 
\begin{definition}\label{defi:D}
Consider a real number $s\ge 0$ and 
a function $\phi\colon [0,\infty)\to (0,\infty)$ satisfying the doubling condition
$\phi(2r)\leq c_0\phi(r)$ for any $r\geq 0$ and some constant $c_0>0$. 
Then we may define $|D|^{s,\phi}$ as the Fourier multiplier with symbol $|\xi|^s\phi(|\xi|)$. More precisely,
\begin{equation*}
\mathcal{F}( |D|^{s,\phi}f)(\xi)=|\xi|^s\phi(|\xi|) \mathcal{F}(f)(\xi).
\end{equation*}
\end{definition}

We shall consider operators $\D^{s,\phi}$ for some special functions $\phi$ depending on some 
function $\kappa\colon [0,\infty)\to (0,\infty)$, of the form
\begin{equation}\label{n10}
\phi(\lam)=\int_{0}^{\infty}\frac{1-\cos(h)}{h^2} \kappa\left(\frac{\lam}{h}\right) \dh, \quad \text{for }\lambda\ge 0.
\end{equation}
Before we explain the reason to introduce these functions, let us clarify the assumptions on $\kappa$ which will be needed later on. 

\begin{assumption}\label{A:kappa}
Throughout this paper, we always require that 
$\kappa\colon[0,\infty) \to [1,\infty)$ satisfies the following three assumptions: 
\begin{enumerate}[$({{\rm H}}1)$]
\item\label{H1} $\kappa$ is increasing and $\lim\kappa(r)=\infty$ when $r$ goes to $+\infty$;
\item\label{H2} there is a positive constant $c_0$ such that $\kappa(2r)\leq c_0\kappa(r)$ for any $r\geq 0$; 
\item\label{H3} the function $r\mapsto \kappa(r)/\log(4+r)$ is decreasing on  $[0,\infty)$.
\end{enumerate}
\end{assumption}
\begin{remark}
$i)$ The main example is the function $\kappa_a(r)=(\log(4+r))^a$ with $a\in [0,1]$. 
However, we shall see that, to include the critical case $a=\frac{1}{3}$ in Theorem~$\ref{theo:main}$, 
we must consider more general functions (see Section~$\ref{S:critical}$). 

$ii)$ To clarify notations, we mention that we choose the natural logarithm so that $\log(e)=1$. 
In assumption~$({\rm H}\ref{H3})$, 
the choice of the constant $4$ is purely technical (it is used only to prove~\e{n97} below). 
\end{remark}

There is one observation
that will be useful below. We will prove that $\phi\sim \kappa$, 
which means that
\be\label{n60}
c\kappa(\lam)\le \phi(\lam)\le C \kappa(\lam),
\ee
for some positive constants $c,C$. In particular, $\phi$ satisfies the doubling condition 
$\phi(2r)\leq c_0\phi(r)$ and we may define the Fourier multiplier $\D^{s,\phi}=\D^{s}\phi(\la D_x\ra)$, 
as in Definition~\ref{defi:D}. 
Although $\kappa$ and $\phi$ are equivalent, 
we will use them for different purposes. 
We will use $\phi$ when we prefer to work with the 
frequency variable (Fourier analysis), while we use $\kappa$ when 
the physical variable is more convenient. 
The next two results will be useful later on to freely switch computations between the 
frequency and physical settings. 

\begin{lemma}\label{Z12}
Assume that $\phi$ is as defined in~\e{n10} for some function $\kappa$ 
satisfying Assumption~$\ref{A:kappa}$. Then, for all 
$g\in \mathcal{S}(\xR)$, there holds
\begin{equation*}
\D^{1,\phi}g(x)=\frac{1}{4}\int_{\xR}\frac{2g(x)-g(x+h)-g(x-h)}{h^2}
\kappa\left(\frac{1}{|h|}\right) \dh.
\end{equation*}
\end{lemma}
\begin{proof}
Notice that the Fourier transform of the function
\begin{align*}
\int_\xR \frac{2g(x)-g(x+h)-g(x-h)}{h^2}\kappa\left(\frac{1}{|h|}\right)\dh,
\end{align*}
is given by
$$
\left(\int_\xR \frac{2-2\cos(h\xi)}{h^2}\kappa\left(\frac{1}{|h|}\right) \dh\right)
\hat g(\xi).
$$
Therefore
$$
\left(4|\xi|  \int_0^\infty \frac{1-\cos(h)}{h^2} \kappa\left(\frac{|\xi|}{|h|}\right)  \dh\right) 
\hat{g}(\xi)=4\phi(\xi) \la\xi\ra\hat{g}(\xi)
=4\widehat{\D^{1,\phi}g}(\xi),
$$
equivalent to the wanted result. 
\end{proof}

The following result states 
that $\phi$ and $\kappa$ are equivalent and also gives the equivalence of some 
semi-norms.
\begin{proposition} \label{Z9}
Assume that $\phi$ is as defined in~\e{n10} for some function $\kappa$ 
satisfying Assumption~$\ref{A:kappa}$.

$i)$ There exist two constants $c,C>0$ such that, for all $\lambda\ge 0$,
\be\label{n60b}
c\kappa(\lam)\le \phi(\lam)\le C \kappa(\lam).
\ee

$ii)$ Given $g\in \mathcal{S}(\xR)$, define the semi-norm
$$
\lA g\rA_{s,\kappa}=\left(\iint_{\xR^2} \la 2g(x)-g(x+h)-g(x-h)\ra^2
\left(\frac{1}{|h|^{s}}\kappa\left(\frac{1}{|h|}\right)\right)^2\frac{\dx\dh}{|h|}\right)^\mez.
$$
Then, for all $1<s<2$, there exist two constants $c,C>0$ such that, for 
all $g\in \mathcal{S}(\xR)$,
\begin{equation*}
c\int_{\xR}  \la \D^{s,\phi}g(x)\ra^2\dx\le \lA g\rA_{s,\kappa}^2
\le C\int_{\xR}\la \D^{s,\phi}g(x)\ra^2\dx.
\end{equation*}
\end{proposition}
\begin{proof}We begin by proving statement $ii)$. Let us introduce
\begin{equation*}
{K}(r)=\kappa^2\left(\frac{1}{r}\right)\frac{1}{r^{1+2s}}\qquad (r>0).
\end{equation*}
For any $h\in\xR$, the Fourier transform of $x\mapsto 2g(x)-g(x+h)-g(x-h)$ 
is given by $(2-2\cos(\xi h))\hat{g}(\xi)$. 
Consequently, 
$$
\lA g\rA_{s,\kappa}^2=\iint_{\xR^2}\left| 2g(x)-g(x+h)-g(x-h)\right|^2{K}(|h|)\dx\dh
=\int_\xR I(\xi)\bla \hat{g}(\xi)\bra^2 \dxi,
$$
where
$$
I(\xi)=\frac{2}{\pi}\int_\xR  ( 1-\cos(\xi h))^2  {K}\left(\la h\ra\right) \dh.
$$
We must prove that the integral $I(\xi)$ satisfies
\begin{equation}
\label{Z30} c|\xi|^{2s}\phi(|\xi|)^2\le I(\xi)\le C|\xi|^{2s}\phi(|\xi|)^2,
\end{equation}
for some constant $c,C$ independent of $\xi\in\xR$. 
Let us prove the bound from above. To do so, we use the inequality 
$\la 1-\cos(\theta)\ra\le \min \{ 2,\theta^2\}$ for all $\theta\in \xR$ to obtain
$$
I(\xi)\le \frac{8}{\pi}\int_{|h \xi|\ge 1}{K}\left(\la h\ra\right) \dh
+\frac{2}{\pi}\int_{|h \xi|\le 1}\xi^4h^4{K}\left(\la h\ra\right) \dh.
$$
Now, since $\kappa$ is increasing by assumption, directly from the definition of ${K}$, we have
$$
\int_{|h \xi|\ge 1}{K}\left(\la h\ra\right) \dh\le \left(\kappa(|\xi|)\right)^2
\int_{|h \xi|\ge 1}\frac{\dh}{|h|^{1+2s}}\les \kappa^2(|\xi|)\la \xi\ra^{2s}.
$$
The estimate of the contribution of the integral over $\{|h \xi|\le 1\}$ is more involved. 
To do so, we introduce the following decomposition of the integrand
\be\label{n105}
\xi^4h^4{K}\left(\la h\ra\right)=\xi^4h^4\kappa^2\left(\frac{1}{\la h\ra}\right)\frac{1}{\la h\ra^{1+2s}}
=\pi_1\left(\la h\ra\right)\pi_2\left(\la h\ra\right)\pi_3\left(\la h\ra\right)\frac{\xi^4}{|h|^{s-1}}
\ee
where
\begin{align*}
\pi_1(r)&\defn \frac{\kappa^2\big(\frac{1}{r}\big)}{\log\big(4+\frac{1}{r}\big)^2},\quad
&&\pi_2(r)\defn\frac{\log\big(4+\frac{1}{r}\big)^2}{\log\big(\lambda_0+\frac{1}{r^2}\big)^2},\\
\pi_3(r)&\defn r^{2-s}\left(\log\left(\lambda_0+\frac{1}{r^2}\right)\right)^{2}, &&\lambda_0=\exp\left(\frac{8 }{2-s}\right).
\end{align*}
By assumption on $\kappa$ (see~$({\rm H}\ref{H3})$ in Assumption~\ref{A:kappa}), the function $\pi_1$ is increasing 
and hence
$$
\pi_1(|h|)\le \frac{\kappa^2(|\xi|)}{\log(4+|\xi|)^2}\quad\text{for}\quad |h|\le \frac{1}{|\xi|}\cdot
$$
The function $\pi_2$ is bounded on $[0,+\infty)$ by some harmless constant depending only on $s$. 
Eventually, we claim that the function $\pi_3$ is increasing. Indeed, 
\begin{align*}
\frac{\diff}{\dr}\pi_3(r)&=r^{1-s}\log\left(\lambda_0+\frac{1}{r^2}\right)^{2}\left[2-s-\frac{4}{(\lambda_0 r^2+1)\log\left(\lambda_0+\frac{1}{r^2}\right)}\right]\\&\geq  r^{1-s}\log\left(\lambda_0+\frac{1}{r^2}\right)^{2}\left(2-s-\frac{4}{\log\left(\lambda_0\right)}\right)\\
&=\frac{2-s}{2}r^{1-s}\log\left(\lambda_0+\frac{1}{r^2}\right)^{2}>0.
\end{align*}
It follows that
$$
\pi_3(|h|)\le |\xi|^{s-2}
\log\left(\lambda_0+|\xi|^2\right)^{2}\quad\text{for}\quad |h|\le \frac{1}{|\xi|}\cdot
$$
By combining these bounds about the factors $\pi_j$, we deduce from~\e{n105} that
$$
\int_{|h \xi|\le 1}\xi^4h^4{K}\left(\la h\ra\right) \dh
\les  \left( \frac{\kappa(|\xi|)}{\log(4+|\xi|)} \right)^2|\xi|^{2+s}
\log\left(\lambda_0+\la\xi\ra^2\right)^{2} \int_{|h|\le 1/|\xi|}\frac{\dh}{|h|^{s-1}}.
$$
Since $\log(4+|\xi|)\sim \log(\lambda_0+|\xi|^2)$ and since $0<s-1<1$, we deduce that
$$
\int_{|h \xi|\le 1}\xi^4h^4{K}\left(\la h\ra\right) \dh\lesssim  (\kappa(|\xi|))^2 |\xi|^{2s}.
$$
So, we get $I(\xi)\lesssim|\xi|^{2s}\phi(|\xi|)^2$. 

On the other hand, 
$$
I(\xi)\gtrsim \int_{\frac{2\pi}{5}\leq \xi h\leq \frac{3\pi}{5}}  {K}\left(\la h\ra\right) \dh
\gtrsim \left(\int_{\frac{2\pi}{5}\leq \xi h\leq \frac{3\pi}{5}} \dh\right)  \kappa^2(|\xi|)|\xi|^{1+2s}
\gtrsim   (\kappa(|\xi|))^2|\xi|^{2s}.
$$
Therefore, we proved \eqref{Z30}, which concludes the proof of statement $ii)$.

It remains to prove statement $i)$. The lower bound $\phi(\lambda)\ge c\kappa(\lambda)$ 
follows directly from the definition of $\phi$, by writing 
$$
\phi(\lambda)\ge \int_0^1  \frac{1-\cos(h)}{h^2}\kappa\left(\frac{\lambda}{h}\right) \dh
\ge \left(\int_0^1\frac{1-\cos(h)}{h^2}\dh\right)\kappa(\lambda),
$$
since $\kappa$ is increasing. To prove the upper bound, we split the integral into 
$\{h\le 1\}$ and $\{h>1\}$ and then use similar arguments to those used above. 
\end{proof}

\subsection{Paralinearization of the nonlinearity}\label{S:2.3}
For a nonlinear evolution equation, considerable insight 
comes from being able to decompose the nonlinearity 
into several pieces having different roles. 
For a parabolic free boundary problem in fluid dynamics, 
one expects to extract from the nonlinearity at least two terms: 
\begin{enumerate}[$i)$]
\item a convective term of the form $V\partial_x f$,
\item an elliptic component of the form $\gamma \D^\alpha f$,
\end{enumerate}
for some coefficients $V$ and $\gamma$ and some index $\alpha\ge 0$, 
where as above $\D=(-\Delta)^{1/2}$ (see Definition~\ref{defi:D} with $s=1$ and $\phi=1$, or~\e{n3} below).
To reach this goal, a standard strategy is use a paradifferential analysis, 
which consists in using a Littlewood-Paley decomposition to 
determine the relative significance of competing terms. 
For the Muskat equation, 
this idea was implemented independently in~\cite{Alazard-Lazar,Nguyen-Pausader}. 
In this paragraph, we recall the approach in~\cite{Alazard-Lazar} 
where the formulation of the Muskat equation 
in terms of finite differences is exploited to 
give such a paradifferential decomposition in a direct manner.

Recall that the Muskat equation reads
$$
\partial_tf=\frac{1}{\pi}\int_\xR\frac{\partial_x\Delta_\alpha f}{1+\left(\Delta_\alpha f\right)^2}\dalpha.
$$
Therefore it can be written under the form
$$
 \partial_tf= \frac{1}{\pi}\int_\xR\partial_x\Delta_\alpha f\dalpha 
 - \frac{1}{\pi}\int_\xR\partial_x\Delta_\alpha f\frac{\left(\Delta_\alpha f\right)^2}{1+\left(\Delta_\alpha f\right)^2}\dalpha.
$$
Let us introduce now some notations that will be used continually 
in the rest of the paper. 
We define the singular integral operators
\be\label{n3}
\mathcal{H}u=-\frac{1}{\pi}\mathrm{pv}\int_\xR\Delta_\alpha u\dalpha\quad\text{and}\quad \D=\mathcal{H}\partial_x.
\ee
Then the Muskat equation can be written under the form
\begin{align}\label{main}
\partial_tf+\D f = \mathcal{T}(f)f,
\end{align}
where $\mathcal{T}(f)$ is the operator defined by
\be\label{def:T(f)f}
\mathcal{T}(f)g = -\frac{1}{\pi}\int_\xR\left(\partial_x\Delta_\alpha g\right)
\frac{\left(\Delta_\alpha f\right)^2}{1+\left(\Delta_\alpha f\right)^2}\dalpha.
\ee
The desired decomposition of the nonlinearity alluded to above will 
be achieved by 
splitting the coefficient
$$
F_\alpha\defn\frac{\left(\Delta_\alpha f\right)^2}{1+\left(\Delta_\alpha f\right)^2}
$$
into its odd and even components. Set 
\begin{align}
&\mathcal{O}\left(\alpha,\cdot\right)
= \frac{1}{2}\frac{\left(\Delta_\alpha f\right)^2}{1+\left(\Delta_\alpha f\right)^2} -\frac{1}{2}\frac{\left(\Delta_{-\alpha} f\right)^2}{1+\left(\Delta_{-\alpha} f\right)^2},\label{Oss}\\&
\mathcal{E}\left(\alpha,\cdot\right) = \frac{1}{2}\frac{\left(\Delta_\alpha f\right)^2}{1+\left(\Delta_\alpha f\right)^2} +\frac{1}{2}\frac{\left(\Delta_{-\alpha} f\right)^2}{1+\left(\Delta_{-\alpha} f\right)^2}\cdot\label{Ess}
\end{align}
It follows that 
\begin{align*}
\mathcal{T}(f)g = -\frac{1}{\pi}\int_\xR\left(\partial_x\Delta_\alpha g\right)\mathcal{E}\left(\alpha,\cdot\right) \dalpha-\frac{1}{\pi}\int_\xR\left(\partial_x\Delta_\alpha g\right)\mathcal{O}\left(\alpha,\cdot\right) \dalpha.
\end{align*}
Since $\Delta_\alpha f(x)$ converges to $f_x(x)$ when $\alpha$ 
goes to $0$, we further 
decompose $\mathcal{E}\left(\alpha,\cdot\right)$ as
$$
\mathcal{E}\left(\alpha,\cdot\right) =\frac{(\partial_xf)^2}{1+(\partial_xf)^2}+\left(\mathcal{E}\left(\alpha,\cdot\right) -\frac{(\partial_xf)^2}{1+(\partial_xf)^2}\right).
$$
Remembering that 
\begin{equation*}
\D g(x)= -\frac{1}{\pi}\int_\xR\left(\partial_x\Delta_\alpha g\right)\dalpha,
\end{equation*}
we obtain the following decomposition of the nonlinearity:
\begin{equation}\label{n1}
\mathcal{T}(f)g=\frac{(\partial_xf)^2}{1+(\partial_xf)^2}\D g+V(f)\partial_x g+R(f,g).
\end{equation}
where 
\be\label{defi:V}
V=-\frac{1}{\pi}\int_\xR\frac{\mathcal{O}\left(\alpha,.\right)}{\alpha}\dalpha,
\ee
and 
\begin{align*}
R(f,g)=-\frac{1}{\pi}\int_\xR\left(\partial_x\Delta_\alpha g\right)\left(\mathcal{E}\left(\alpha,\cdot\right) -\frac{(\partial_xf)^2}{1+(\partial_xf)^2}\right)\dalpha+\frac{1}{\pi}\int_\xR\frac{\partial_x g(.-\alpha)}{\alpha}\mathcal{O}\left(\alpha,\cdot\right) \dalpha.
\end{align*}

\subsection{Nonlinear estimates}\label{S:2.4}
With these preliminaries established, we start the analysis of the 
nonlinearity in the Muskat equation. 

We begin by estimating the coefficient $V(f)$ (see~\e{defi:V}). 

\begin{proposition} \label{Z3'}
There exists a positive constant $C$ such that, for all $f$ in $\mathcal{S}(\xR)$, 
\begin{equation}\label{Z3}
\lA V(f)\rA_{\dot H^{1}}\leq C\lA f\rA_{\dot{H}^2}+C\lA f\rA_{\dot{H}^{\frac{7}{4}}}^2.
\end{equation}
\end{proposition}
\begin{proof}
Recall that
$$
V=-\frac{1}{\pi}\int_\xR\frac{\mathcal{O}\left(\alpha,.\right)}{\alpha}\dalpha 
\quad\text{where}\quad
\mathcal{O}\left(\alpha,\cdot\right) = \frac{1}{2}\frac{\left(\Delta_\alpha f\right)^2}{1+\left(\Delta_\alpha f\right)^2} -\frac{1}{2}\frac{\left(\Delta_{-\alpha} f\right)^2}{1+\left(\Delta_{-\alpha} f\right)^2}.
$$
Now, write 
\begin{equation}\label{Z100}
\mathcal{O}\left(\alpha,\cdot\right) =  	A_\alpha(x)\left(\Delta_\alpha f(x)-\Delta_{-\alpha} f(x)\right),
\end{equation}
where
\begin{equation*}
A_\alpha(x)= \frac{1}{2}
\frac{\Delta_\alpha f+\Delta_{-\alpha} f}{(1+\left(\Delta_\alpha f\right)^2)(1+\left(\Delta_{-\alpha} f\right)^2)}.
\end{equation*}
Hence
$$
\partial_x\mathcal{O}=
A_\alpha \big(\Delta_\alpha f_x-\Delta_{-\alpha} f_x\big)
+\partial_xA_\alpha\big(\Delta_\alpha f-\Delta_{-\alpha} f\big).
$$
Now, we replace in the first product the factor $A_\alpha$ by 
$$
\left(A_\alpha(x)-\frac{f_x(x)}{(1+f_x(x)^2)^2}\right)+\frac{f_x(x)}{(1+f_x(x)^2)^2},
$$
and observe that the last term is bounded by $1$. Therefore, 
by using the triangle inequality, it follows that
\begin{align*}
&\la\partial_x V(x)\ra\le I_1(x)+I_2(x)+I_3(x)\quad\text{where}\\[1ex]
&I_1(x)=\la\int_\xR\left(\Delta_\alpha f_x(x)-\Delta_{-\alpha} f_x(x)\right)\frac{\dalpha}{\alpha}\ra,\\
&I_2(x)= \int_\xR \left|A_\alpha(x)-\frac{f_x(x)}{(1+f_x(x)^2)^2}\right|\left|\Delta_\alpha f_x(x)-\Delta_{-\alpha} f_x(x)\right|\frac{\dalpha}{|\alpha|},\\
&I_3(x)= \int_\xR \la\partial_xA_\alpha(x)\ra \la\Delta_\alpha f(x)-\Delta_{-\alpha} f(x)\ra\frac{\dalpha}{|\alpha|}.
\end{align*}
We now must estimate the $L^2$-norm of $I_j$ for $1\le j\le 3$. 

We begin by estimating the $L^2$-norm of $I_1$. Observe that
$$
\Delta_\alpha f_x(x)-\Delta_{-\alpha} f_x(x)=
\frac{2f_x(x)-f_x(x-\alpha)-f_x(x+\alpha)}{\alpha}.
$$
Now, as above, we use the fact that, for any $\alpha\in\xR$, the 
Fourier transform of $x\mapsto 2g(x)-g(x+\alpha)-g(x-\alpha)$ 
is given by $(2-2\cos(\xi \alpha))\hat{g}(\xi)$. 
Consequently, it follows from 
Plancherel's theorem that
\begin{align*}
\lA I_1\rA_{L^2}^2&=\frac{1}{2\pi}\int_\xR |\xi|^2 \left|\int\left( 2 -e^{-i\alpha \xi}-e^{i\alpha \xi}\right)\frac{\dalpha}{\alpha^2}\right|^2|\hat{f}(\xi)|^2 \dxi \\&=  \frac{2}{\pi}\int_\xR |\xi|^2 \left|\int\left( 1 -\cos(\alpha \xi)\right)\frac{\dalpha}{\alpha^2}\right|^2|\hat{f}(\xi)|^2 \dxi
\\&=  \frac{2}{\pi}\left|\int\left( 1 -\cos(\alpha)\right)\frac{\dalpha}{\alpha^2}\right|^2\int_\xR |\xi|^4|\hat{f}(\xi)|^2 \dxi.
\end{align*}
Now, since $\la 1 -\cos(\alpha )\ra \leq \min\{2,|\alpha\xi|^2\}$,
\begin{align*}
 \left|\int_\xR\left( 1 -\cos(\alpha )\right)\frac{\dalpha}{\alpha^2}\right|\leq 2\int_{|\alpha|\geq 1}\frac{\dalpha}{\alpha^2}
 +\int_{|\alpha|\leq 1}|\alpha|^2\frac{\dalpha}{\alpha^2}\leq 3.
\end{align*}
So, 
\begin{align*}
\lA I_1\rA_{L^2}^2\les \int_\xR |\xi|^4|\hat{f}(\xi)|^2 dx=\lA f\rA_{\dot{H}^2(\mathbb{R})}^2.
\end{align*}
This implies the wanted inequality $\lA I_1\rA_{L^2}\les   \lA f\rA_{\dot{H}^2}$. 

We now move to the estimate of $\lA I_2\rA_{L^2}$.  Introduce
\be\label{defi:F}
F(x,y)=\mez \frac{x+y}{(1+x^2)(1+y^2)},
\ee
so that
\be\label{n20}
A_\alpha=F(\Delta_\alpha f,\Delta_{-\alpha}f)\quad\text{and}\quad \frac{f_x(x)}{(1+f_x(x)^2)^2}=F(f_x,f_x).
\ee
Since $\lA\nabla_{x,y}F\rA_{L^\infty(\xR^2)}\le 4$, we deduce that
\begin{align*}
\la A_\alpha(x)-\frac{f_x(x)}{(1+f_x(x)^2)^2}\ra\le 4 \la \Delta_{\alpha}f-f_x\ra+ 4\la \Delta_{-\alpha}f-f_x\ra,
\end{align*}
which in turn implies that, for all $x$ in $\xR$,
$$
\la I_2(x)\ra
\le 8 \int_\xR \left|\Delta_\alpha f(x)-f_x(x)\right|\left|\Delta_\alpha f_x(x)-\Delta_{-\alpha} f_x(x)\right|\frac{\dalpha}{|\alpha|},
$$
where we used the change of variables $\alpha\mapsto -\alpha$ to handle the contribution of the term $\Delta_{-\alpha}f-f_x$. Now, using the obvious estimate 
$$
\left|\Delta_\alpha f_x(x)-\Delta_{-\alpha} f_x(x)\right|\le \left|\Delta_\alpha f_x(x)\right|+\left|\Delta_{-\alpha} f_x(x)\right|,
$$
and the Cauchy-Schwarz inequality, we conclude that
\be\label{n16}
\lA I_2\rA_{L^2}^2
\les\left( \int_\xR\left(\int _\xR(\Delta_\alpha f-f_x(x))^2\frac{\dalpha}{\alpha^2}\right)^2 \dx\right)^{\mez}
\left( \int_\xR\left(\int_\xR \left(\Delta_\alpha f_x(x)\right)^2\dalpha\right)^2 \dx\right)^{\mez}.
\ee
We next claim that the first factor in the right-hand side above 
can be estimated by the second one. To see this, 
we start with the identity
$$
\Delta_\alpha f-f_x(x)=\frac{f(x)-f(x-\alpha)}{\alpha}-f_x(x)=\frac{1}{\alpha}\int_0^{\alpha}(f_x(x-y)-f_x(x))\dy.
$$
Then, for all $x$ and all $\alpha$ in $\xR$, we deduce that
$$
(\Delta_\alpha f-f_x(x))^2\le \frac{1}{\alpha}\la \int_0^{\alpha}(f_x(x-y)-f_x(x))^2\dy\ra.
$$
(Notice that the integral is not necessarily non-negative 
since $\alpha$ might be non-positive.) 
Then, integrating in $\alpha$ and splitting the integral 
into $\alpha\ge 0$ and $\alpha<0$, we obtain
\begin{align*}
\int_\xR \left(\Delta_\alpha f-f_x(x)\right)^2\frac{\dalpha}{\alpha^2}
&\leq	\int_{\xR}\la\int_0^\alpha \left( f_x(x-y)-f_x(x)\right)^2\dy\ra\frac{\dalpha}{|\alpha|^3}\\
&\leq2\int_{0}^\infty\int_0^\alpha \left( f_x(x-y)-f_x(x)\right)^2\dy\frac{\dalpha}{\alpha^3}.
\end{align*}
Since $2\int_y^{+\infty}\alpha^{-3}\dalpha\le y^{-2}$, using Fubini's theorem, 
it follows that
\be\label{Z1}
\begin{aligned}
\int_\xR \left(\Delta_\alpha f-f_x(x)\right)^2\frac{\dalpha}{\alpha^2}
&\leq  \int_0^\infty (f_x(x-y)-f_x(x))^2\frac{\dy}{y^2}\\
&\leq \int_0^\infty (\Delta_y f_x(x))^2 \dy\leq \int_\xR \left(\Delta_\alpha f_x(x)\right)^2\dalpha.
\end{aligned}
\ee
Hence, it follows from~\e{n16} that
$$
\lA I_2\rA_{L^2}^2\les 
\int_\xR\left(\int_\xR \left(\Delta_\alpha f_x(x)\right)^2\dalpha\right)^2 \dx.
$$
Now, using the Triebel-Lizorkin semi-norms~\e{n5}, observe that
\be\label{n120}
\begin{aligned}
\int_\xR \left( \int_\xR \la\Delta_\alpha f_x\ra^2\dalpha\right)^2\dx
&=\bigg(\int_\xR\bigg( \int_\xR \la\delta_\alpha f_x\ra^2\frac{\dalpha}{|\alpha|^{1+2\mez}}\bigg)^\frac{4}{2}\dx\bigg)^{\frac{4}{4}}\\
&=\lA f_x\rA^4_{\dot{F}^{\mez}_{4,2}}\les \lA f_x\rA_{\dot H^{\frac{3}{4}}}^4,
\end{aligned}
\ee
where we used the Sobolev embedding~\e{FB} in the last inequality. 
This proves that $\lA I_2\rA_{L^2}$ is estimated by $C\lA f\rA_{\dot H^{7/4}}^2$, 
which completes the analysis of $I_2$. 

It remains to estimate $\lA I_3\rA_{L^2}$. Remembering that the function $F$ in~\e{defi:F} 
has partial derivatives bounded on $\xR^2$, we find that
$$
\la \partial_xA_\alpha\ra\les \la \Delta_\alpha f_x\ra+\la \Delta_{-\alpha}f_x\ra.
$$
Now, making the change of variable $\alpha\mapsto-\alpha$ and applying 
the Cauchy-Schwarz inequality, we end up with
\begin{align*}
I_3(x)&\les \int_\xR\la \Delta_\alpha f_x(x)\ra \la\Delta_\alpha f(x)-\Delta_{-\alpha} f(x)\ra
\frac{\dalpha}{|\alpha|}\\
&\les \left(\int_\xR (\Delta_\alpha f_x(x))^2\dalpha\right)^{\mez}
\left(\int_\xR \left(\Delta_\alpha f(x)-\Delta_{-\alpha} f(x)\right)^2\frac{\dalpha}{\alpha^2}\right)^{\mez}.
\end{align*}
Now, using the inequality
$$
\left(\Delta_\alpha f(x)-\Delta_{-\alpha} f(x)\right)^2\le 2\left(\Delta_\alpha f(x)-f_x(x)\right)^2+2\left(\Delta_{-\alpha}f(x)-f_x(x)\right)^2,
$$
and then applying~\eqref{Z1},  we infer that
\begin{equation*}
\int_\xR \left(\Delta_\alpha f(x)-\Delta_{-\alpha} f(x)\right)^2\frac{\dalpha}{\alpha^2}\les
\int_\xR (\Delta_\alpha f_x(x))^2\dalpha.
\end{equation*}
Consequently, 
$$
\lA I_3\rA_{L^2}^2\les  \int_\xR\left( \int_\xR\left(\Delta_\alpha f_x(x)\right)^2\dalpha\right)^2\dx.
$$
Now, it follows from~\e{n120} that 
\begin{equation*}
\lA I_3\rA_{L^2}^2\les  \lA f\rA_{\dot H^{\frac{7}{4}}}^4,
\end{equation*}
which  concludes the proof of the proposition.
\end{proof}

\begin{remark} It follows from \eqref{Z100} that for all $f$ in $\mathcal{S}(\xR)$,
\begin{align*}
|V(f)(x)|&\le \frac{1}{2\pi}\int |\Delta_{\alpha}f(x)-\Delta_{-\alpha}f(x)|\frac{\dalpha}{|\alpha|}\\
&\le \frac{1}{\pi}  \iint_{\xR^2} |1-\cos(\alpha\xi)||\hat f(\xi)
\frac{\dxi \dalpha}{|\alpha|^2}.
\end{align*}
Thus, since $\int_{\mathbb{R}} |1-\cos(\alpha\xi)|\frac{ \dalpha}{|\alpha|^2}=|\xi|\int_{\mathbb{R}} |1-\cos(\alpha)|\frac{ \dalpha}{|\alpha|^2}\sim |\xi|$, we find that
\begin{equation}\label{Z101}
\lA V(f)\rA_{L^\infty}\lesssim \int|\xi| |\hat f(\xi)|\dxi. 
\end{equation}
We will use this estimate to bound the $H^1(\xR_x)$-norm of 
$\mathcal{T}(f)g$ in Corollary~\ref{bounTfg}.
\end{remark}

The next result contains a key estimate which will allow us to commute arbitrary operators with~$\mathcal{T}(f)$.

\begin{proposition}\label{Z18}
Assume that $\phi$ is as defined in~\e{n10} for some function $\kappa$ 
satisfying Assumption~$\ref{A:kappa}$. Then, there exists a positive constant $C$ such that, for all $f,g$ in $\mathcal{S}(\xR)$, 
\begin{align}\label{n18}
\blA\big[\D^{1,\phi},\mathcal{T}(f)\big](g)\brA_{L^2}&\lesssim \lA g\rA_{\dot H^{\frac74}}\blA \D^{\frac74,\phi}f\brA_{L^2}
+\lA g\rA_{\dot H^{\frac74}}\blA \D^{\frac{7}{4},\phi^2}f\brA_{L^2}^{\mez}\lA f\rA_{\dot{H}^{\frac{19}{12}}}^{\tdm}\\&\quad+\blA\D^{\frac{7}{4},\phi^{2}}g\brA_{L^2}^{\mez} \lA g\rA_{\dot H^{\frac74}}^{\mez} \lA f\rA_{\dot H^{\frac74}}.\nonumber
\end{align}
\end{proposition}
\begin{proof}
Recall that the operator $\mathcal{T}(f)$ is defined by
$$
\mathcal{T}(f)g
 = -\frac{1}{\pi}\int_\xR\left(\Delta_\alpha g_x\right)F_\alpha\dalpha\quad\text{where}\quad
 F_\alpha=\frac{\left(\Delta_\alpha f\right)^2}{1+\left(\Delta_\alpha f\right)^2}\dalpha.
$$
Let us introduce 
$$
\Gamma_\alpha\defn\D^{1,\phi}\left[F_\alpha\Delta_\alpha g_x \right]-F_\alpha\D^{1,\phi}
\left[\Delta_\alpha g_x\right]-\Delta_\alpha g_x\D^{1,\phi}\left[F_\alpha\right].
$$
Loosely speaking, the term $\Gamma_\alpha$ 
is a remainder term for a fractional Leibniz rule with the operator $\D^{1,\phi}$. 
With this notation, one can write the commutator of the operators $\D^{1,\phi}$ and $\mathcal{T}(f)$ as
$$
\left[\D^{1,\phi},\mathcal{T}(f)\right](g)= -\frac{1}{\pi} \int_\xR \Delta_\alpha g_x\D^{1,\phi}F_\alpha\dalpha
-\frac{1}{\pi} \int_\xR  \Gamma_\alpha\dalpha.
$$
Consequently, to estimate the $L^2$-norm of $\left[\D^{1,\phi},\mathcal{T}(f)\right](f)$, 
we have to bound  the  following integrals:
$$
(I)\defn \int_\xR\left(	\int_\xR \Delta_\alpha g_x(x)\D^{1,\phi}F_\alpha(x)\dalpha\right)^2\dx,\quad 
(II)\defn
\int_\xR\left(	\int_\xR \Gamma_\alpha\dalpha\right)^2\dx.
$$
More precisely, to prove the wanted estimate~\e{n18}, it is sufficient to show that 
\begin{align}
(I)&\lesssim 
\lA g\rA^2_{\dot H^{\frac74}}\blA \D^{\frac74,\phi}f\brA_{L^2}^2
+\lA g\rA^2_{\dot H^{\frac74}}\blA \D^{\frac{7}{4},\phi^2}f\brA_{L^2}\blA \D^{\frac{19}{12}}f\brA_{L^2}^3,\label{Z11}\\
(II)&\lesssim  \blA\D^{\frac{7}{4},\phi^{2}}g\brA_{L^2}\blA\D^{\frac{7}{4}}g\brA_{L^2}
\blA\D^{\frac{7}{4}}f\brA_{L^2}^2.\label{Z13}
\end{align}

\textit{Step 1:} We prove \eqref{Z11}.  By Holder's inequality and Minkowski's inequality, one has 
$$
(I)	\leq \left(\int_\xR\left(	\int_\xR \big\vert\Delta_\alpha g_x(x)\big\vert^2 \dalpha\right)^2\dx\right)^{\mez}
\left(\int_\xR\left(\int_\xR \big\vert\D^{1,\phi}F_\alpha(x)\big\vert^2 \dalpha\right)^2dx\right)^{\mez}.
$$
The first factor is estimated by means of~\e{n120}, namely
$$
\left(\int_\xR\left(\int_\xR \big\vert\Delta_\alpha g_x(x)\big\vert^2 \dalpha\right)^2\dx\right)^{\mez}\les  \lA g\rA^2_{\dot H^{\frac{7}{4}}}.
$$
The analysis of the second term is more difficult. 
We begin by applying Minkowski's inequality together with the Sobolev 
embedding $\dot{H}^{1/4}(\xR)\hookrightarrow L^4(\xR)$, to obtain
\begin{align*}
\left(\int_\xR\left(\int _\xR\big\vert\D^{1,\phi}F_\alpha(x)\big\vert^2 \dalpha\right)^2\dx\right)^{\mez}
&\les \int_\xR\left(	\int_\xR \big\vert\D^{1,\phi}F_\alpha(x)\big\vert^4\dx\right)^{\mez}\dalpha\\
&\les\iint_{\xR^2} \big\vert\D^{\frac54,\phi}F_\alpha(x)\big\vert^2\dx\dalpha.
\end{align*}
Now, to evaluate the latter integral, we use Lemma \ref{Z9}, which implies that
\begin{multline*}
\iint_{\xR^2} \big\vert\D^{\frac54,\phi}F_\alpha(x)\big\vert^2\dx\dalpha\\
\sim  \iint_{\xR^2}  \left| 2F_\alpha( x)-F_\alpha( x+h)-F_\alpha( x-h)\right|^2 \left(\kappa\left(\frac{1}{|h|}\right)\right)^2\frac{\dx\dh}{|h|^{1+5/2}}\cdot
\end{multline*}
We must estimate the integrand $\left| 2F_\alpha( x)-F_\alpha( x+h)-F_\alpha( x-h)\right|^2$ 
in terms of similar terms for $\Delta_\alpha f$. To do so, write $F_\alpha=
\mathcal{F}(\Delta_\alpha)$ with $\mathcal{F}(x)=x^2/(1+x^2)$. Then 
we use the following sharp contraction estimate for $\mathcal{F}$. 

\begin{lemma}\label{Z10}
For any triple of real numbers $(x_1,x_2,x_3)$, there holds
	\begin{equation}\label{Z4}
	\left|\frac{2x_1^2}{1+x_1^2}-\frac{x_2^2}{1+x_2^2}-\frac{x_3^2}{1+x_3^2}\right|\leq |x_2+x_3-2x_1| +|x_2-x_1|^2+|x_3-x_1|^2.
	\end{equation}	
\end{lemma}
\begin{remark}
It is worth remarking that such a clean inequality does not hold for a general function.
\end{remark}
\begin{proof}
Write
\begin{align*}
&\frac{2x_1^2}{1+x_1^2}-\frac{x_2^2}{1+x_2^2}-\frac{x_3^2}{1+x_3^2}
=\frac{2}{1+x_1^2}-\frac{1}{1+x_2^2}-\frac{1}{1+x_3^2}\\[1ex]
&=\frac{(x_3^2+1)(x_2+x_1)(x_2-x_1)+(x_2^2+1)(x_3+x_1)(x_3-x_1)}{(1+x_1^2)(1+x_2^2)(1+x_3^2)}\\[1ex]
&=\frac{(x_3^2+1)(x_2+x_1)(x_2+x_3-2x_1)+(x_1x_2+x_2x_3+x_3x_1-1)(x_2-x_3)(x_3-x_1)}{(1+x_1^2)(1+x_2^2)(1+x_3^2)},
\end{align*}
to obtain
\begin{align*}
\left|\frac{2x_1^2}{1+x_1^2}-\frac{x_2^2}{1+x_2^2}-\frac{x_3^2}{1+x_3^2}\right|
\leq \vert x_2+x_3-2x_1\vert +\vert x_2-x_3\vert \vert x_3-x_1\vert,
\end{align*} 
which implies the desired result.
\end{proof}

Directly from the definition of $F_\alpha$, it follows from the previous lemma that
\begin{align*}
\left| 2F_\alpha( x)-F_\alpha( x+h)-F_\alpha( x-h)\right|^2
&\les\left| 2\Delta_\alpha f( x)-\Delta_\alpha f( x+h)-\Delta_\alpha f( x-h)\right|^2 \\
&\quad+\left| \Delta_\alpha f( x)-\Delta_\alpha f( x+h)\right|^4\\
&\quad+\left| \Delta_\alpha f( x)-\Delta_\alpha f( x-h)\right|^4.
\end{align*}
Integrating in $h$ and making use of the change of variable $h\mapsto -h$ 
to handle the contribution of the last term, 
we conclude that
\begin{align*}
&\int_{\xR}  \left| \D^{5/4,\phi}\left[F_\alpha\right](x)\right|^2\dx\\
&\qquad\qquad\lesssim 
\iint_{\xR^2}  \left| 2\Delta_\alpha f( x)-\Delta_\alpha f( x+h)-\Delta_\alpha f( x-h)\right|^2
\left(\kappa\left(\frac{1}{|h|}\right)\right)^2\frac{\dx\dh}{|h|^{1+5/2}}
\\&\qquad\qquad\quad+
\iint_{\xR^2}  \left| \Delta_\alpha f( x)-\Delta_\alpha f( x+h)\right|^4
\left(\kappa\left(\frac{1}{|h|}\right)\right)^2\frac{ \dx\dh}{|h|^{1+5/2}}.
 \end{align*}
The first term is estimated by using again Lemma \ref{Z9}, which implies that
\begin{multline*}
 \iint_{\xR^2}  \left| 2\Delta_\alpha f( x)-\Delta_\alpha f( x+h)-\Delta_\alpha f( x-h)\right|^2
\left(\kappa\left(\frac{1}{|h|}\right)\right)^2\frac{\dx\dh}{|h|^{1+5/2}}\\
  \les  \int_{\xR}  \left| \D^{5/4,\phi}\left[\Delta_\alpha f\right](x)\right|^2\dx.
 \end{multline*}
On the other hand, using the same arguments together with Holder's inequality, we have
\begin{align*}
 &\iint_{\xR^2}  \left| \Delta_\alpha f( x)-\Delta_\alpha f( x+h)\right|^4 \left(\kappa\left(\frac{1}{|h|}\right)\right)^2\frac{\dx\dh}{|h|^{1+5/2}}\\
 &\qquad\qquad\leq  \left(\iint_{\xR^2}  \left| \Delta_\alpha f( x)-\Delta_\alpha f( x+h)\right|^2\left(\kappa\left(\frac{1}{|h|}\right)\right)^4\frac{\dx\dh}{|h|^{1+5/2}}\right)^{1/2} \\
 &\qquad\qquad\quad \times\left(\iint_{\xR^2}  \left| \Delta_\alpha f( x)-\Delta_\alpha f( x+h)\right|^6\frac{\dx\dh}{|h|^{1+5/2}}\right)^{1/2}\\
 &\qquad\qquad\sim \blA\D^{\frac{5}{4},\phi^2}\Delta_\alpha f\brA_{L^2} \lA\Delta_\alpha f\rA_{\dot{F}^{\frac{5}{12}}_{6,6}}^3
\\&\qquad\qquad \lesssim \blA\D^{\frac{5}{4},\phi^2}\Delta_\alpha f\brA_{L^2} \lA\Delta_\alpha f\rA_{\dot{H}^{\frac{3}{4}}}^3,
\end{align*}
where we used the Sobolev embedding $\dot{H}^{\frac{3}{4}}(\xR)\hookrightarrow \dot{F}^{\frac{5}{12}}_{6,6}(\xR)$, see~\e{FB}. Therefore, by gathering the previous results, we get
$$
(I)\lesssim   \lA g\rA^2_{\dot H^{\frac74}}\left(\iint_{\xR^2}  \bla\Delta_\alpha( \D^{\frac54,\phi}f)(x)\bra^2 \dx \dalpha
+\int_{\xR}  \blA\D^{\frac{5}{4},\phi^2}\Delta_\alpha f\brA_{L^2} \lA \Delta_\alpha f\rA_{\dot{H}^{\frac{3}{4}}}^3\dalpha\right).
$$
Now we claim that, for any function $\tilde{f}$, 
\be\label{n15}
\iint_{\xR^2}  \bla \Delta_\alpha \tilde{f}\bra^2\dalpha\dx\sim 
\blA \tilde{f}\brA_{\dot{H}^{\mez}}^2.
\ee
To see this, write
$$
\iint_{\xR^2}  \bla \Delta_\alpha \tilde{f}\bra^2\dalpha\dx=
\iint_{\xR^2}  \bla \delta_\alpha \tilde{f}\bra^2\frac{\dalpha}{\alpha^{1+2\mez}}\dx
=\blA \tilde{f}\brA_{\dot{F}_{2,2}^{\mez}}^2
\sim \blA \tilde{f}\brA_{\dot{H}^{\mez}}^2,
$$
where we used~\e{SE0}. Now the estimate~\e{n15} implies that
\begin{align*}
& \int\bla\Delta_\alpha( \D^{\frac54,\phi}f)(x)\bra^2 \dx \dalpha\les \blA \D^{\frac74,\phi}f\brA_{L^2}^2 ,\\
&\int_{\xR}  \brA\Delta_\alpha (\D^{\frac{5}{4},\phi^2}f)\brA_{L^2}^2\dalpha
\les \blA \D^{\frac{7}{4},\phi^2}f\brA_{L^2}^2.
\end{align*}
It follows that
$$
(I)\lesssim 
\lA g\rA^2_{\dot H^{\frac74}}\blA \D^{\frac74,\phi}f\brA_{L^2}^2 +\lA g\rA^2_{\dot H^{\frac74}}\blA \D^{\frac{7}{4},\phi^2}f\brA_{L^2}\left(\int_{\xR}\brA\Delta_\alpha \big( \D^{\frac{3}{4}}f\big)\brA_{L^2}^6 \dalpha\right)^{\mez}.
$$
Using the Besov norm~\e{defi:B-spaces}, we have
$$
\left(\int_{\xR}\brA\Delta_\alpha \big( \D^{\frac{3}{4}}f\big)\brA_{L^2}^6 \dalpha\right)^{\mez}
=\blA \D^{\frac34}f\brA_{\dot{B}^{\frac56}_{2,6}}^3\les \blA \D^{\frac34}f\brA_{\dot{F}^{\frac56}_{2,6}}^3
\les \blA \D^{\frac{19}{12}}f\brA_{L^2}^3,
$$
where we have used the embedding~\e{FB} in the last inequality, while the inner inequality follows 
at once from the definitions of Besov and Triebel-Lizorkin space (see~\e{defi:B-spaces} and \e{n5}) by using the Minkowski's inequality. This proves the wanted result~\eqref{Z11}.

\textit{Step 2:} We prove \e{Z13}. 
Recall that Lemma~\ref{Z12} implies that for any function $g$, 
one can compute $\D^{1,\phi}g$ as follows:
\begin{equation*}
\D^{1,\phi}g(x)=\frac{1}{4}\int\frac{2g(x)-g(x+h)-g(x-h)}{h^2}\kappa\left(\frac{1}{|h|}\right)\dh.
\end{equation*}
Then using the elementary identity
\begin{align*}
&\left(2xy-x_{1}y_1-x_{-1}y_{-1}\right)-x	\left(2y-y_1-y_{-1}\right)-y	\left(2x-x_{1}-x_{-1}\right)\\&=-(x-x_1)(y-y_1)-(x-x_{-1})(y-y_{-1}),
\end{align*}
we deduce that
\begin{align*}
\la \Gamma_\alpha\ra=2\left|\int_\xR \left(F_\alpha(x)-F_\alpha(x-h)\right)
\left(\Delta_\alpha g_x(x)-\Delta_\alpha g_x(x-h)\right) \kappa\left(\frac{1}{|h|}\right)\frac{\dh}{h^2}\right|.
\end{align*}
Since
\begin{equation*}
|F_\alpha(x)-F_\alpha(x-h)|\leq |\Delta_\alpha f(x)-\Delta_\alpha f(x-h)|,
\end{equation*}
it follows that
\begin{align*}
\la\Gamma_\alpha\ra&\leq 2\int_{\xR}  \left|\Delta_\alpha f(x)-\Delta_\alpha f(x-h)\right|\left|\Delta_\alpha g_x(x)-\Delta_\alpha g_x(x-h)\right|
\kappa\left(\frac{1}{|h|}\right)\frac{\dh}{h^2}
\\
&\leq 2 
\left(\int_{\xR}  \left|\Delta_\alpha g_x(x)-\Delta_\alpha g_x(x-h)\right|^2 \kappa^4\left(\frac{1}{|h|}\right)\frac{\dh}{h^2}\right)^{\frac14}\\
&\times \left(\int_{\xR}  \la\Delta_\alpha g_x(x)-\Delta_\alpha g_x(x-h)\ra^3 \frac{\dh}{h^2}\right)^{\frac 1 6} 
\left(\int_{\xR}\left|\Delta_\alpha f(x)-\Delta_\alpha f(x-h)\right|^{\frac{12}{7}} \frac{\dh}{h^2}\right)^{\frac{7}{12}}.
\end{align*}
So, by Holder's inequality in $x$,
\begin{align*}
\lA \Gamma_\alpha\rA_{L^2(\xR;\dx)}^2&\leq 4 \left(\iint_{\xR^2}\left(\Delta_\alpha g_x(x)-\Delta_\alpha g_x(x-h)\right)^2 \kappa^4\left(\frac{1}{|h|}\right)
\frac{\dh}{h^2}\dx\right)^{\mez}\\
&\quad\times \left(\iint_{\xR^2}\left|\Delta_\alpha g_x(x)-\Delta_\alpha g_x(x-h)\right|^3 \frac{\dh\dx}{h^2}\right)^{\frac13}\\
&\quad\times \left(\int_{\xR}\left(\int_{\xR}\left|\Delta_\alpha f(x)-\Delta_\alpha f(x-h)\right|^{\frac{12}{7}} \frac{\dh}{h^2}\right)^{7}\dx\right)^{\frac16}\\
&\sim \blA |\Delta_\alpha (\D^{\frac{3}{2},\phi^{2}}g)\brA_{L^2}
\lA\Delta_\alpha g_x\rA_{\dot{F}^{\frac13}_{3,3}}
\lA\Delta_\alpha f\rA_{\dot{F}^{\frac{7}{12}}_{12,\frac{12}{7}}}^2\\&\overset{\eqref{FB}}\lesssim  
\blA \Delta_\alpha (\D^{\frac{3}{2},\phi^{2}}g)\brA_{L^2}
\blA\Delta_\alpha (\D^{\tdm} g)\brA_{L^2}
\blA\Delta_\alpha (\D f)\brA_{L^2}^2.
\end{align*}
Therefore,
\begin{align*}
(II)&\leq  \left(\int_{\xR} \blA |\Delta_\alpha (\D^{\frac{3}{2},\phi^{2}}g)\brA_{L^2}^{1/2}
\blA\Delta_\alpha (\D^{3/2} g)\brA_{L^2}^{1/2} \blA\Delta_\alpha (\D f)\brA_{L^2}\dalpha\right)^2
\\
&\leq \left(\int_{\xR}\blA\Delta_\alpha (\D^{\frac{3}{2},\phi^{2}}g)\brA_{L^2}^2
|\alpha|^{1/2}\dalpha\right)^{1/2}
\left(\int_{\xR}\blA\Delta_\alpha (\D^{\frac{3}{2}}g)\brA_{L^2}^2|\alpha|^{1/2}\dalpha\right)^{1/2}\\
&\quad\times \int_{\xR}\blA\Delta_\alpha (\D f)\brA_{L^2}^2|\alpha|^{-1/2}\dalpha\\
&\sim  \blA\D^{\frac{7}{4},\phi^{2}}g\brA_{L^2}\blA\D^{\frac{7}{4}}g\brA_{L^2}\blA\D^{\frac{7}{4}}f\brA_{L^2}^2,
\end{align*}
where we used~\e{SE0}. This gives \eqref{Z13}, which completes the proof.
\end{proof}

Eventually, we study the remainder term $R(f,g)$ in the 
paralinearization of $\mathcal{T}(f)g$ (see~\e{n1}).
\begin{proposition}\label{Z19}
Assume that $\phi$ is as defined in~\e{n10} for some function $\kappa$ satisfying Assumption~$\ref{A:kappa}$. 
Then, there exists a positive constant $C$ such that, for all 
functions~$f,g$ in $\mathcal{S}(\xR)$,
\begin{equation}\label{Z14}
\Vert R(f,g)\Vert _{L^2}\le C\Vert g\Vert _{\dot{H}^{\frac{3}{4}}}  \lA f\rA_{\dot{H}^{\frac{7}{4}}}.
\end{equation}
In particular, 
\begin{equation}\label{Z15}
\Vert R(f,\D^{1,\phi}f)\Vert _{L^2}\le  C\blA \D^{\frac{7}{4},\phi}f\brA_{L^2} \lA f\rA_{\dot{H}^{\frac{7}{4}}}.
\end{equation}
\end{proposition}
\begin{proof} Recall that 
\begin{align*}
R(f,g)&=-\frac{1}{2\pi}\int_\xR\left(\partial_x\Delta_\alpha g+\partial_x\Delta_{-\alpha} g\right)\left(\mathcal{E}\left(\alpha,\cdot\right) -\frac{(\partial_xf)^2}{1+(\partial_xf)^2}\right)\dalpha\\
&\quad+\frac{1}{\pi}\int_\xR\frac{\partial_x g(.-\alpha)}{\alpha}\mathcal{O}\left(\alpha,\cdot\right) \dalpha.
\end{align*}
The first half of the proof is based on~\cite{Alazard-Lazar}. Namely, write
\begin{align*}
&\partial_x\Delta_\alpha g+\partial_x\Delta_{-\alpha} g=-\frac{1}{\alpha}\partial_\alpha \big(\delta_\alpha g+\delta_{-\alpha}g\big),\\
&\partial_x g(.-\alpha)=\partial_\alpha(\delta_\alpha g),
\end{align*}
as can be verified by elementary calculations. 
Then use these identities to integrate by parts in $\alpha$. By so doing, we find
\begin{align*}
|R(f,g)(x)|\lesssim\int_\xR \frac{|\delta_\alpha g(x)|}{|\alpha|}&\bigg(\left|\mathcal{E}\left(\alpha,x\right) -\frac{(\partial_xf(x))^2}{1+(\partial_xf(x))^2}\right|+|\alpha| |\partial_\alpha\mathcal{E}\left(\alpha,x\right) |\\
&~~~+|\mathcal{O}(\alpha,x)|+|\alpha||\partial_\alpha\mathcal{O}\left(\alpha,x\right) | \bigg)\frac{\dalpha}{|\alpha|}\cdot
\end{align*}
Then it follows from \cite[Lemma 4.5]{Alazard-Lazar} that
$$
|R(f,g)(x)|\lesssim \int_\xR |\Delta_\alpha g(x)|\gamma(\alpha,x)\dalpha
$$
where 
$$
\gamma(\alpha,x)=\frac{1}{\la\alpha\ra}\left(
 |\delta_\alpha f_x(x)|+ |\delta_{-\alpha} f_x(x)|+\frac{|s_\alpha f(x)|}{|\alpha|}+\left|\frac{1}{\alpha}\int_{0}^{\alpha}s_\eta f_x(x)\deta\right|\right).
$$
and $s_\alpha f(x)=\delta_\alpha f(x)+\delta_{-\alpha} f(x)$.\\
We now use different arguments then those used in~\cite{Alazard-Lazar}. 
The main new ingredient here is given by the following inequality:
$$
\int_\xR \gamma(\alpha,x)^2\dalpha\les \int_\xR  |\delta_\alpha f_x(x)|^2\frac{\dalpha}{\alpha^2}\cdot
$$
To prove the latter, it will suffice to show that
that
\begin{align}
&\int_\xR \left|\int_{0}^{\alpha}s_\eta f_x(x)\deta\right|^2\frac{\dalpha}{\alpha^4}
\les \int_\xR  |\delta_\alpha f_x(x)|^2\frac{\dalpha}{\alpha^2},\label{n100}\\
&\int_\xR  |s_\alpha f(x)|^2\frac{\dalpha}{\alpha^4}\lesssim \int_\xR  |\delta_\alpha f_x(x)|^2\frac{\dalpha}{\alpha^2}\cdot\label{n101}
\end{align}
To prove \eqref{n100}, we apply the following 
Hardy's inequality
$$
\int_0^\infty\left(\frac{1}{\alpha}\int_0^\alpha u(\eta)\deta\right)^2\dalpha\le 4 \int_0^\infty u(\alpha)^2\dalpha,
$$
with $u(\eta)=(s_\eta f_x(x))/\eta$. It follows that
\begin{align*}
\int_\xR \left|\int_{0}^{\alpha}s_\eta f_x(x)\deta\right|^2\frac{\dalpha}{\alpha^4}
&=\int_\xR \left|\frac{1}{\alpha^2}\int_{0}^{\alpha}s_\eta f_x(x)\deta\right|^2\dalpha
\le \int_\xR \left|\frac{1}{\alpha}\int_{0}^{\alpha}\frac{s_\eta f_x(x)}{\eta}\deta\right|^2\dalpha\\
&\le 4 \int_\xR \left( \frac{s_\alpha f_x(x)}{\alpha}\right)^2\dalpha.
\end{align*}
Then~\e{n100} follows at once from the triangle inequality 
$\la s_\alpha f\ra^2\le 2\la \delta_\alpha f\ra^2+2\la \delta_{-\alpha} f\ra^2$.

Let us prove \e{n101}. Write
$$
\frac{s_\alpha f}{\alpha}=\frac{\delta_\alpha f+\delta_{-\alpha}f}{\alpha}=\Delta_\alpha f-\Delta_{-\alpha}f
=\big(\Delta_\alpha f-f_x\big)-\big(\Delta_{-\alpha}f-f_x\big),
$$
so
$$
  \frac{(s_\alpha f(x))^2}{\alpha^4}\le 2\frac{(\Delta_\alpha f-f_x)^2}{\alpha^2}+2\frac{(\Delta_{-\alpha}f-f_x)^2}{\alpha^2}.
$$
Now, remember from~\e{Z1} that 
$$
\int_\xR \left(\Delta_\alpha f-f_x(x)\right)^2\frac{\dalpha}{\alpha^2}\leq  \int_\xR \left(\Delta_\alpha f_x(x)\right)^2\dalpha,
$$
together with a similar result for $\Delta_{-\alpha}f-f_x$ (interchanging $\alpha$ and $-\alpha$). This proves~\e{n101}. 

Then, remembering that $\Delta_\alpha u=\delta_\alpha u/\alpha$ and using the Cauchy-Schwarz inequality, we get
\begin{equation*}
|R(f,g)(x)|\lesssim
\left(\int_{\xR}  (\Delta_{\alpha} g(x))^2 \dalpha\right)^{\mez}\left(\int_{\xR} 
(\Delta_\alpha f_x(x))^2\dalpha\right)^{\mez}.
\end{equation*}
By using the Cauchy-Schwarz inequality again, 
we conclude that
$$
\Vert R(f,g)\Vert _{L^2}^2
\lesssim  \left(\int_{\xR}\left(\int_\xR (\Delta_\alpha g(x))^2 \dalpha\right)^2\dx\right)^{\mez}
\left( \int_{\xR}\left(\int_{\xR}(\Delta_\alpha f_x(x))^2\dalpha\right)^2\dx\right)^{\mez}.
$$
Therefore, the estimate~\e{n120} implies that
$$
\Vert R(f,g)\Vert _{L^2}^2
\lesssim \Vert g\Vert _{\dot{H}^{\frac{3}{4}}}^2  \lA f\rA_{\dot{H}^{\frac{7}{4}}}^2,
$$
equivalent to the desired result~\eqref{Z14}.
\end{proof}
\begin{corollary}\label{bounTfg}
There exists a positive constant $C$ such that, for all~$f\in \mathcal{S}(\xR)$,
\begin{equation}\label{Z102}
||\mathcal{T}(f)f||_{\dot H^1}\le C\left(\lA f\rA_{\dot H^{\frac32}}+\lA f\rA_{\dot H^{\frac32}}^2+1+||V(f)||_{L^\infty}\right) \lA f\rA_{\dot H^{2}}.
\end{equation}
\end{corollary}
\begin{proof}
By \eqref{n1} and \eqref{n18} and \eqref{Z15} with $\phi\equiv 1$ one has, 
\begin{align*}
\Vert \mathcal{T}(f)f\Vert_{\dot H^1}&\leq 	\blA\big[\D,\mathcal{T}(f)\big]f\brA_{L^2}+	
\Vert\mathcal{T}(f)(|D|f)\Vert_{L^2}\\
&\lesssim \lA f\rA_{\dot H^{\frac74}}^2
+\lA f\rA_{\dot H^{\frac74}}^{\frac{3}{2}}\lA f\rA_{\dot{H}^{\frac{19}{12}}}^{\frac{3}{2}}+\left(||V(f)||_{L^\infty}+1\right)||f||_{\dot H^2}.
\end{align*}
Then \eqref{Z102} follows from the classical interpolation inequalities in Sobolev spaces.
\end{proof}

\section{Proof of the main results}\label{S:3}
In this section, we prove Theorem~\ref{theo:main} and Theorem~\ref{theo:main2}. 

Following a classical strategy, we shall construct solutions of the Muskat equation in three steps: 
\begin{enumerate}
\item We begin by defining approximate 
systems and proving that the Cauchy problem for the latter are 
well-posed by means of an ODE argument.
\item Secondly, we prove uniform estimates for the solutions of the approximate systems 
on a uniform time interval. The heart of the entire argument is contained in a {\em a priori} estimate 
given by Proposition~\ref{P:3.3} below.
\item Finally, we prove that the sequence of approximate solutions 
converges to a solution of the Muskat equation and conclude the proof by proving a 
uniqueness result. 
\end{enumerate}

\textbf{Notations.} We denote by $\langle\cdot,\cdot\rangle$ the scalar product in $L^2(\xR)$ and 
set $\lA \cdot\rA=\lA \cdot\rA_{L^2}$. 
 \subsection{Approximate systems}

We will define the wanted approximate systems by using a 
version of Galerkin's method based on Friedrichs mollifiers. 
To do so, it is convenient to use smoothing operators 
which are also projections. Consider, for any integer $n$ in $\xN\setminus\{0\}$, the operators $J_n$ defined by 
\be\label{defi:Jn}
\begin{aligned} 
\widehat{J_n u}(\xi)&=\widehat{u}(\xi) \quad &&\text{for} &&\la \xi\ra\le n,\\
\widehat{J_n u}(\xi)&=0 \quad &&\text{for} &&\la \xi\ra> n.
\end{aligned}
\ee
Notice that $J_n$ is a projection since $J_n^2=J_n$.

Recall that the Muskat equation reads
$$
\partial_tf=\frac{1}{\pi}\int_\xR\frac{\partial_x\Delta_\alpha f}{1+\left(\Delta_\alpha f\right)^2}\dalpha.
$$
Remember also from~\S\ref{S:2.3} that the latter is equivalent to 
\be\label{n80}
\partial_tf+\D f = \mathcal{T}(f)f,
\ee
where $\mathcal{T}(f)$ is the operator defined by~\e{def:T(f)f}.

Let us introduce the following approximate Cauchy problems:
\begin{equation}\label{A3}
\left\{
\begin{aligned}
&\partial_t f_n+\D f_n=J_n\big(\mathcal{T}(f_n)f_n\big),\\
& f_n\arrowvert_{t=0}=J_n f_{0}.
\end{aligned}
\right.
\end{equation} 
The next lemma states that this system has smooth global in time solutions.
\begin{lemma}\label{L:3.1}
For all $f_0\in L^{2}(\xR)$, and any $n\in \xN\setminus\{0\}$, the initial value 
problem~\e{A3} has a unique global solution 
$f_n\in C^{1}([0,+\infty);H^{\infty}(\xR))$. 
Moreover
\be\label{n52}
f_n=J_n f_n,
\ee
and, for all time $t\ge 0$,
\be\label{n51}
\lA f_n(t)\rA_{L^2}\le \lA f_0\rA_{L^2}.
\ee
\end{lemma}
\begin{proof}This proof is not new: it follows from the analysis in~\cite[Section 5]{Alazard-Lazar} together with the $L^2$-maximum principle 
in \cite[Section 2]{CCGRPS-JEMS2013}. However, since slight modifications are needed, we include a detailed proof. 

$i)$ We begin by studying the following auxiliary Cauchy problem
\begin{equation}\label{A3-twisted}
\left\{
\begin{aligned}
&\partial_t f_n+J_n\D f_n=J_n\big(\mathcal{T}(J_n f_n)J_n f_n\big),\\
&f\arrowvert_{t=0}=J_n f_{0}.
\end{aligned}
\right.
\end{equation} 
The Cauchy problem \eqref{A3-twisted} has the form
\begin{equation}\label{edo}
\partial_t f_n= F_n(f_n),\quad f_n\arrowvert_{t=0}=J_nf_0,
\end{equation}
where 
$$
F_n(f)=-\D J_n f+J_n\big(\mathcal{T}(J_nf)J_nf\big)
$$
Recall from Proposition~$2.3$ in~\cite{Alazard-Lazar} that 
the map $f\mapsto \mathcal{T}(f)f$ is locally Lipschitz from 
$\dot{H}^1(\xR)\cap \dot{H}^{\tdm}(\xR)$ to $L^2(\xR)$. 
Therefore, since $J_n$ is linear smoothing operator (which means that it is bounded from 
$L^2(\xR)$ into $H^\mu(\xR)$ for any $\mu\ge 0$), the map $f\mapsto J_n\big(\mathcal{T}(J_nf)J_nf\big)$ 
is locally Lipschitz from 
$L^2(\xR)$ to $L^2(\xR)$. This implies that $F_n$ satisfies the same property and hence 
we are thus in position to apply 
the Cauchy-Lipschitz theorem. This gives 
the existence of a unique maximal solution~$f_n$ 
in~$C^{1}([0,T_n);L^{2}(\xR))$. 
Moreover, the 
continuation principle for ordinary differential equations implies that either
\be\label{A4}
T_n=+\infty\qquad\text{or}\qquad \limsup_{t\rightarrow T_n} \lA f_n(t)\rA_{L^2}=+\infty.
\ee
We shall prove in the next step that $T_n=+\infty$. 

Eventually, remembering that $J_n^2=J_n$, we check that the function $(I-J_n)f_n$ solves
$$
\partial_t (I-J_n)f_n=0,\quad (I-J_n)f_n\arrowvert_{t=0}=0.
$$
Therefore $(I-J_n)f_n=0$ which proves that $J_nf_n=f_n$. 

Now, we deduce from $J_nf_n=f_n$ and the equation~\e{A3-twisted} that 
$f_n$ is also a solution to the original equation~\e{A3}. 
In addition, the identity $J_nf_n=f_n$ also implies that $f_n$ is smooth, in particular $f_n$ belongs to~$C^{1}([0,T_n);H^{\infty}(\xR))$. 

$ii)$ To conclude the proof of the proposition, it remains to show that 
$(i)$ the solution is defined globally in time and $(ii)$ 
it satisfies the $L^2$-bound
\be\label{n50}
\lA f_n(t)\rA_{L^2}\le \lA f_0\rA_{L^2}. 
\ee
In fact, in light of the alternative~\e{A4}, it is sufficient to prove the latter inequality: 
by combining~\e{A4} with~\e{n50}, we will obtain that $T_n=+\infty$.

It remains to prove \e{n50}. This estimate is proved in \cite[Section 2]{CCGRPS-JEMS2013} 
for the full equation (that is with $J_n$ replaced by the identity $I$) and we recall the main 
argument to verify that the estimate is uniform in $n$. By definition of $\mathcal{T}(f)f$, one has 
$$
\D f-\mathcal{T}(f)f=\frac{1}{\pi}\int_\xR\frac{\partial_x\Delta_\alpha f}{1+\left(\Delta_\alpha f\right)^2}\dalpha.
$$
Therefore the equation~\e{A3} is equivalent to
$$
\partial_tf_n+(I-J_n)\D f_n  =J_n\left(\frac{1}{\pi}\int_\xR\frac{\partial_x\Delta_\alpha f_n}{1+\left(\Delta_\alpha f_n\right)^2}\dalpha\right).
$$
Using $f_n$ as test function, one has 
\begin{equation*}
\frac{1}{2}\fract \lA f_n(t)\rA_{L^2}^2+\langle (I-J_n)\D f_n, f_n\rangle
= \frac{1}{\pi}\bigg\langle J_n\int_\xR\frac{\partial_x\Delta_\alpha f_n}{1+\left(\Delta_\alpha f_n\right)^2}\dalpha , f_n\bigg\rangle.
\end{equation*}
Now we use three elementary ingredients: firstly, $\langle (I-J_n)\D f_n, f_n\rangle\ge 0$ and $J_n^*=J_n$ as can be verified by applying Plancherel's theorem. and secondly $J_nf_n=f_n$. It follows that
$$
\frac{1}{2}\fract \lA f_n(t)\rA_{L^2}^2\le
\frac{1}{\pi}\bigg\langle \int_\xR\frac{\partial_x\Delta_\alpha f_n}{1+\left(\Delta_\alpha f_n\right)^2}\dalpha , f_n\bigg\rangle.
$$
Now, by \cite[Section 2]{CCGRPS-JEMS2013}, the right-hand side is non-positive since, for any smooth function $f=f(t,x)$, 
\begin{multline*}
\int_\xR\left[\int_\xR\frac{\partial_x\Delta_\alpha f}{1+\left(\Delta_\alpha f\right)^2}\dalpha \right]f(x)\dx
\\
=-\iint_{\xR^2}\log\left[\sqrt{1+\frac{(f(t,x)-f(t,x-\alpha))^2}{\alpha^2}}\right]\dx\dalpha.
\end{multline*}
The proof is complete.
\end{proof}

\subsection{Uniform estimates}\label{S:3.2}
We have seen that the solutions~$f_n$ to the approximate systems~\e{A3} satisfy a uniform $L^2$-estimate (see~\e{n51}). 
We now have to prove uniform $L^2$-estimate 
for the derivatives $\D^{s,\phi}f_n$. 

Let us fix some notations. 
\begin{assumption}\label{A:kappa2}
We consider a 
function $\kappa\colon[0,\infty) \to [1,\infty)$ satisfying Assumption~\ref{A:kappa} together with the following property:
There exists $a\in [0,1/2)$ such that
$$
\kappa(r)\ge \log (4+r)^a\quad \text{for all}\quad r\ge 0.
$$
\end{assumption}
Remember that, by notation,
\begin{equation*}
\phi(r)=\int_{0}^{\infty}\frac{1-\cos(h)}{h^2} \kappa\left(\frac{r}{|h|}\right) \dh, \quad \text{for }r\ge 0.
\end{equation*}
Recall also that $\phi$ and $\kappa$ are equivalent: 
there are $c,C>0$ such that,
$$
\forall r\ge 0,\qquad c\kappa(r)\le \phi(r)\le C \kappa(r).
$$
We denote by $\D^{s,\phi}$ the Fourier multiplier $\D^{s}\phi(\la D_x\ra)$. 

With this notations, our goal in this paragraph is to obtain uniform estimates 
for the functions 
\be\label{n67}
A_n(t)=\blA \D^{\tdm,\phi}f_n(t)\brA_{L^2}^2, \qquad B_n(t)=\blA \D^{2,\phi}f_n(t)\brA_{L^2}^2.
\ee
The following result is the key technical point in this paper.
\begin{proposition}\label{P:3.3}
Assume that $\kappa$ satisfies Assumptions~$\ref{A:kappa}$ and~$\ref{A:kappa2}$. Then 
there exist two positive constants $C_1$ and $C_2$ such that, for all integer $n\in \xN\setminus\{0\}$,
\be\label{Z21'}
\fract A_n(t)+C_1\delta_n(t)B_n(t)\leq C_2 \left( \sqrt{A_n(t)}+A_n(t) \right)\mu_n(t) B_n(t),
\ee
where
\begin{align*}
\delta_n(t)&=\left(1+ \log\left(4+\frac{B_n(t)}{A_n(t)+ \Vert f_0\Vert_{L^2}^2}\right)^{1-2a}\left( A_n(t)+ \Vert f_0\Vert_{L^2}^2\right)\right)^{-1},\\
\mu_n(t)&=\left(\kappa\left(\frac{B_n(t)}{A_n(t)}\right)\right)^{-1}.
\end{align*}
\end{proposition}
\begin{proof}
We split the analysis into two parts:
\begin{enumerate}
\item We begin by applying the nonlinear estimates proved in Section~\ref{S:2} to deduce a key inequality (see~\e{Z21}) of the form
\be\label{Z21ter}
\fract \blA \D^{\tdm,\phi}f_n\brA_{L^2}^2
+ \int_\xR \frac{\bla\D^{2,\phi}f_n\bra^2}{1+(\partial_x f_n)^2}\dx\le C Q(f_n) \blA\D^{2,\phi}f_n\brA_{L^2},
\ee
where $Q(f_n)$ is bounded in (strict) subspace of $L^2_{t,x}$ by 
$\blA \D^{\tdm,\phi}f_n\brA_{L^\infty_t(L^2_x)}$ and $\blA \D^{2,\phi}f_n\brA_{L^2_t(L^2_x)}$. 
\item Then we apply interpolation type arguments to show that one can absorb the right-hand side of \e{Z21ter} by the left-hand side.
\end{enumerate}

We now proceed to the details and begin with the following result.
\begin{lemma}\label{L:3.4}
There exists a positive constant $C$ such that, for any $n\in\xN\setminus\{0\}$, the approximate 
solution $f_n\in C^{1}([0,+\infty);H^{\infty}(\xR))$ to~\e{A3} satisfies 
\be\label{Z21}
\fract \blA \D^{\tdm,\phi}f_n\brA_{L^2}^2
+ \int_\xR \frac{\bla\D^{2,\phi}f_n\bra^2}{1+(\partial_x f_n)^2}\dx\le C Q(f_n) \blA\D^{2,\phi}f_n\brA_{L^2},
\ee
where
\begin{align*}
Q(f_n)&= \left(\lA f_n\rA_{\dot H^2}+\lA f_n\rA_{\dot H^{\frac{7}{4}}}^2\right) 
\blA\D^{\tdm,\phi}f_n\brA_{L^2}
+\blA\D^{\frac74,\phi}f_n\brA_{L^2}
\lA f_n\rA_{{H}^{\frac74}}\\
&\quad+\left(\lA f_n\rA_{H^{\frac{19}{12}}}^{3/2}+\lA f_n\rA_{\dot H^{\frac74}}^{1/2}\right) \blA\D^{\frac{7}{4},\phi^{2}}f_n\brA^{1/2}_{L^2}
\lA f_n\rA_{\dot H^{\frac74}}.
\end{align*} 
\end{lemma}
\begin{proof}
As we have seen in the proof of Lemma~$\ref{L:3.1}$, $f_n$ satisfies $J_nf_n=f_n$ and hence 
$f_n\in C^1([0,+\infty);H^\infty(\xR))$. 
In particular, all computations below are easily justified. 

The proof is based on the nonlinear estimates established in the previous section, 
together with parabolic energy estimates for the Muskat equation, and a commutator estimate with the Hilbert transform.

We multiply the equation
$$
\partial_t f_n+\D f_n=J_n\mathcal{T}(f_n)f_n,
$$
by $\D^{3,\phi^2}f_n$ and use the following consequences of the Plancherel's identity:
\begin{align*}
&\big\langle\partial_t f_n,\D^{3,\phi^2} f_n\big\rangle=\frac{1}{2}\fract \blA \D^{3/2,\phi}f_n\brA_{L^2}^2,\\
&\big\langle \D f_n,\D^{3,\phi^2} f_n\big\rangle=\blA \D^{2,\phi}f_n\brA_{L^2}^2.
\end{align*}
Now, we need four elementary ingredients:
$$
J_n^*=J_n,\quad J_nf_n=f_n \quad (\text{see}~\e{n52}),\quad 
\D^{3,\phi^2}=\D^{2,\phi}\D^{1,\phi},\quad \big(\D^{1,\phi}\big)^*=\D^{1,\phi}.
$$
Then we easily verify that
\begin{align*}
\big\langle J_n\mathcal{T}(f_n)f_n,\D^{3,\phi^2} f_n\big\rangle&=\big\langle\mathcal{T}(f_n)f_n,J_n\D^{3,\phi^2} f_n\big\rangle
=\big\langle\mathcal{T}(f_n)f_n,\D^{3,\phi^2} J_nf_n\big\rangle\\
&=\big\langle\D^{1,\phi} \mathcal{T}(f_n)f_n,\D^{2,\phi}f_n\big\rangle.
\end{align*}
It follows that
\begin{align*}
\frac{1}{2}\fract \blA \D^{3/2,\phi}f_n\brA_{L^2}^2+ \blA \D^{2,\phi}f_n\brA_{L^2}^2
=\big\langle\D^{1,\phi} \mathcal{T}(f_n)f_n,\D^{2,\phi}f_n\big\rangle.
\end{align*}
Notice that this identity no longer involves the operator $J_n$, which explains that the subsequent estimates are independent of $n$.

Now we commute the operators $\D^{1,\phi}$ and $\mathcal{T}(f_n)$ in the last term, and then 
expand the term $\mathcal{T}(f_n)(\D^{1,\phi}f_n)$ using~\e{n1}. This gives
\begin{align*}
&\frac{1}{2}\fract \blA D^{3/2,\phi}f_n\brA_{L^2}^2
+ \int_\xR \frac{\bla \D^{2,\phi}f_n\bra^2}{1+(\partial_x f_n)^2} \dx=
(I)+(II)+(III)\quad \text{where}\\[1ex]
&(I)\defn\big\langle V(f_n)\partial_x \D^{1,\phi}f_n, \D^{2,\phi}f_n\big\rangle,\\[1ex]
&(II)\defn\big\langle R(f_n,\D^{1,\phi}f_n), \D^{2,\phi}f_n\big\rangle,\\[1ex]
&(III)\defn\Big\langle\left[\D^{1,\phi},\mathcal{T}(f_n)\right]f_n, \D^{2,\phi}f_n\Big\rangle.
\end{align*}

It follows from Propositions \ref{Z18} and~\ref{Z19} that the terms $(II)$ and $(III)$ are estimated by the right-hand side of~\e{Z21}. 
So it remains only to estimate the term $(I)$. To do so, we claim that
\be\label{com1}
(I)\leq C\lA V(f_n)\rA_{\dot{H}^1}\blA \D^{2,\phi}f_n\brA_{L^2}\blA \D^{\tdm,\phi}f_n\brA_{L^2}.
\ee
Assume that this claim is true. Then it will follow from~\e{com1} and Proposition~\ref{Z3'} that $(I)$ is bounded by the right-hand side 
of~\e{Z21}, which will in turn complete the proof.

Now we must prove~\e{com1}. 
We begin by making appear a commutator structure. To do so, 
we notice that, since $\partial_x=-\mathcal{H}\D$, one can rewrite the term $A$ under the form
$$
\langle V(f_n)\partial_x \D^{1,\phi}f_n, \D^{2,\phi}f_n\rangle=-\langle V(f_n)\mathcal{H} \D^{2,\phi}f_n, \D^{2,\phi}f_n\rangle.
$$
We then use $\mathcal{H}^*=-\mathcal{H}$ to infer that
\begin{align}\nonumber
(I)&= -\frac{1}{2}\Big\langle V(f_n)\mathcal{H} \D^{2,\phi}f_n, \D^{2,\phi}f_n\Big\rangle
+\frac{1}{2} \Big\langle \mathcal{H} \big(V(f_n)\D^{2,\phi}f_n\big),\D^{2,\phi}f_n\Big\rangle\\
&= \frac{1}{2} \Big\langle \left[\mathcal{H}, V(f_n)\right ] \D^{2,\phi}f_n,\D^{2,\phi}f_n\Big\rangle.\label{Z106}
\end{align}
Consequently, to prove \e{com1}, it will be sufficient to establish that
\be\label{n2}
\lA \left[\mathcal{H}, V(f_n)\right ] \D^{2,\phi}(f_n)\rA_{L^2}\les\lA V(f_n)\rA_{\dot{H}^1}\blA \D^{3/2,\phi}f_n\brA_{L^2}.
\ee
The latter inequality will be deduced from a commutator estimate of independent interest. We claim that
\begin{equation}\label{Z2}
\lA \left[\mathcal{H}, g_1\right ](\partial_x g_2)\rA_{L^2}
\leq C\lA g_1\rA_{\dot H^{1}}\lA g_2\rA_{\dot H^{\mez}}.
\end{equation}
Notice that the wanted estimate~\e{n2} follows from~\e{Z2} applied with $g_1=V(f)$ and $g_2=\mathcal{H}\D^{1,\phi}$ (since 
$\D^{2,\phi}=\partial_x \mathcal{H}\D^{1,\phi}$).

It remains to prove the commutator estimate~\e{Z2}. Start from the definition of the Hilbert transform (see \e{n3}) and observe that
$$
\lA\left[\mathcal{H}, g_1\right ](\partial_x g_2)\rA_{L^2}^2
=\frac{1}{\pi^2}\int_\xR\left(\int_\xR \frac{g_1(x)-g_1(y)}{x-y}\partial_y(g_2(x)-g_2(y)) \dy\right)^2 \dx.
$$
Integrating by parts in $y$, this gives
\be\label{n4}
\begin{aligned}
\lA\left[\mathcal{H}, g_1\right ](\partial_x g_2)\rA_{L^2}^2
&\lesssim 
\int_\xR \left(\int _\xR\frac{|g_1(x)-g_1(y)\Vert g_2(x)-g_2(y)|}{|x-y|^2}\dy\right)^2\dx\\
&\quad+\int_\xR \left(\int_\xR \frac{\partial_yg_1(y)}{x-y}(g_2(x)-g_2(y)) \dy\right)^2\dx.
\end{aligned}
\ee
Using the Cauchy-Schwarz inequality, we estimate the first term in the right-hand side of \e{n4} 
by,
$$
\left( \int_\xR \left(\int_\xR \frac{|g_1(x)-g_1(y)|^2}{|x-y|^{1+3/2}}\dy\right)^2\dx\right)^{\mez}
\left( \int_\xR \left(\int_\xR \frac{|g_2(x)-g_2(y)|^2}{|x-y|^{1+1/2}}\dy\right)^2\dx\right)^{\mez}.
$$
Using the Lizorkin-Triebel norms introduced in~\e{n5}, the above product 
is in turn estimated from above by 
$$
\lA g_1\rA_{\dot{F}^{\frac34}_{4,2}}^2 \lA g_2\rA_{\dot{F}^{\frac14}_{4,2}}^2.
$$
Now, the Sobolev embedding \e{FB} imply that the right-hand side above is bounded by the right-hand side of \e{Z2}, namely
$$
\lA g_1\rA_{\dot{F}^{\frac34}_{4,2}} \lA g_2\rA_{\dot{F}^{\frac14}_{4,2}}\les 
\lA g_1\rA_{\dot H^{1}}\lA g_2\rA_{\dot H^{\mez}}.
$$
It remains to estimate the second term in the right-hand side of \e{n4}. We use again H\"older's inequality 
to estimate the later term by
$$
 \left(\int_\xR\left(\int_\xR\frac{|\partial_yg_1(y)|^{3/2}}{|x-y|^{1-1/4}} \dy\right)^2 \dx\right)^{\frac{2}{3}}
\left( \int_\xR\left(\int_\xR\frac{|g_2(x)-g_2(y)|^3}{|x-y|^{1+1/2}}\dy\right)^2 \dx\right)^{\frac{1}{3}}.
$$
Then, again, we use~\e{n5} and \e{FB} to estimate the above quantity from above by
$$
\blA \D^{-\frac{1}{4}}\big(|\partial_yg_1|^{\tdm}\big)\brA_{L^2}^{\frac{4}{3}} \Vert g_2\Vert _{\dot{F}^{\frac{1}{6}}_{6,3}}^2
\lesssim \Vert |g_1\Vert _{\dot H^{1}}^2\Vert g_2\Vert _{\dot H^{\mez}}^2.
$$
Here we have used the fact that
\begin{equation*}
\blA \D^{-\frac{1}{4}}h\brA_{L^2}\le 
C \lA h\rA_{L^{\frac{4}{3}}},~~\forall h\in L^{\frac{4}{3}}(\xR). 
\end{equation*}
This completes the proof of~\eqref{Z2} and hence the proof of the lemma. 
\end{proof}

We now continue with the interpolation arguments alluded to previously. 
We want to estimate the various norms which appear in $Q(f)$ in terms of $A_n$ and $B_n$. 
Indeed, given Lemma~\ref{L:3.4}, the proof of Proposition~\ref{P:3.3} reduces to establishing the following result.

\begin{lemma}\label{L:3.5}
Consider a real number $7/4\le s\leq 2$. Then there exists a positive constant $C$ such that, for all 
$n\in \xN\setminus\{0\}$ and for all $t\ge 0$,
\begin{align}
&\lA f_n(t)\rA_{\dot H^s}\le C\mu_n(t)A_n(t)^{2-s}  B_n(t)^{s-\tdm},\label{Z20'}\\
&\blA \D^{\frac{7}{4},\phi^{2}}f_n\brA_{L^2}\leq \mu_n(t)^{-1} A_n(t)^{\uq}B_n(t)^{\uq},\label{n110}
\end{align}
and moreover,
\be
\lA\partial_x f_n(t)\rA_{L^\infty}\le C \log\left(4+\frac{B_n(t)}{A_n(t)+ \lA f_0\rA_{L^2}^2}\right)^{\frac{1-2a}{2}}\left( A_n(t)^{\mez}
+ \lA f_0\rA_{L^2}\right).\label{linf'}
\ee
\end{lemma}
\begin{proof}
For ease of reading, we skip the indexes $n$. 

$i)$ Let $\lambda >0$. 
By cutting the frequency space into low and high frequencies, at the frequency threshold $\la \xi\ra=\lambda$, we obtain
\begin{align*}
\lA f\rA_{\dot H^s}^2&\les \int_\xR |\xi|^{2s}|\hat{f}|^2 \dxi=\int_{|\xi|\le \lambda}|\xi|^{2s}|\hat{f}|^2 \dxi
+\int_{|\xi|> \lambda} |\xi|^{2s}|\hat{f}|^2 \dxi\\
&\les\int_{|\xi|\le \lambda} \frac{\la \xi\ra^{2s-3}}{\kappa(|\xi|)^2}|\xi|^{3}
\phi(|\xi|)^2|\hat{f}|^2 \dxi
+\int_{|\xi|> \lambda} \frac{\la \xi\ra^{2s-4}}{\kappa(|\xi|)^2}|\xi|^{4}\phi(|\xi|)^2|\hat{f}|^2 \dxi,
\end{align*}
where we have used the equivalence $\phi\sim\kappa$. Now Plancherel's theorem implies that
\be\label{n92}
\begin{aligned}
&\int_{|\xi|\le \lambda} |\xi|^{3}\phi(|\xi|)^2|\hat{f}|^2 \dxi\les 
\blA \D^{\tdm,\phi}f\brA_{L^2}^2,\\ 
&\int_{|\xi|> \lambda} |\xi|^{4}\phi(|\xi|)^2|\hat{f}|^2 \dxi\les 
\blA \D^{2,\phi}f\brA_{L^2}^2.
\end{aligned}
\ee
On the other hand, we claim that
\begin{alignat}{4}
&(i)\quad &&\frac{\la \xi\ra^{2s-4}}{\kappa(|\xi|)^2}\le  \frac{\lambda^{2s-4}}{\kappa(\lambda)^2}\qquad &&\text{for}\quad&&\la\xi\ra\ge \lambda,\\
&(ii)&& \frac{\la \xi\ra^{2s-3}}{\kappa(|\xi|)^2}\le  \frac{\lambda^{2s-3}}{\kappa(\lambda)^2} &&\text{for}\quad 
&&\la\xi\ra\le \lambda.\label{n97}
\end{alignat}
The first claim follows directly from the facts that $\kappa$ is increasing and the assumption $s\le 2$ 
(which implies that $2s-4\le 0$). 
To prove the second claim, write
$$
\frac{r^{2s-3}}{\kappa(r)^2}=\frac{ r^{2s-3}}{\log(4+r)^{2}}\times\frac{\log(4+r)^{2}}{\kappa(r)^2}\cdot
$$
By assumption, $\log(4+r)/\kappa(r)$ is increasing. On the other hand, by computing the derivative, we 
verify that the other factor is also an increasing function (since we assume that $s\ge 7/4$).  
It follows that $r\mapsto r^{2s-3}/\kappa(r)^2$ is also increasing which implies the second claim. 
It follows that
$$
\lA f\rA_{\dot H^s}^2\lesssim \lambda^{2s-3} (\kappa(\lambda))^{-2} \Vert \D^{3/2,\phi}f\Vert_{L^2}^2+\lambda^{2s-4}  (\kappa(\lambda))^{-2} \Vert\D^{2,\phi}f\Vert_{L^2}^2.
$$
Chose $\lambda=\Vert \D^{2,\phi}(f)\Vert _{L^2}^2/\Vert \D^{3/2,\phi}(f)\Vert _{L^2}^2$ to obtain
\begin{equation*}
\lA f\rA_{\dot H^s}^2\lesssim \left(\kappa\left(B_n/A_n\right)\right)^{-2}
A_n^{4-2s}B_n^{2s-3},
\end{equation*}
equivalent to the wanted result \e{Z20'}. \\
\newline
$ii)$ As above, for  $\lambda>0$ one has, 
\begin{align*}
\Vert\D^{\frac{7}{4},\phi^{2}}f_n\Vert_{L^2}^2&=\int |\xi|^{\frac{7}{2}}\phi(|\xi|)^4 |\hat f|^2\dxi\\
&\lesssim \lam^{\frac{1}{2}}\phi(\lam)^2 \int_{|\xi|\leq \lambda} |\xi|^{3}\phi(|\xi|)^2 |\hat f|^2\dxi\\
&\quad+\lam^{-\frac{1}{2}}\phi(\lam)^2 \int_{|\xi|\leq \lambda} |\xi|^{4}\phi(|\xi|)^2 |\hat f|^2\dxi.
 \end{align*}
Since $\phi\sim\kappa$, we deduce that 
$$
\blA \D^{\frac{7}{4},\phi^{2}}f_n\brA_{L^2}^2
\lesssim \lam^{1/2}\kappa(\lam)^2
\blA \D^{3/2,\phi}f\brA _{L^2}^2
+\lam^{-1/2}\kappa(\lam)^2 \blA \D^{2,\phi}f\brA_{L^2}^2.
$$
Now take
$$
\lam=\frac{\blA\D^{2,\phi}f\brA_{L^2}^2}{\blA \D^{\tdm,\phi}f\brA_{L^2}^2},
$$
to get the wanted result~ \eqref{n110}.

$iii)$ Starting from the inverse Fourier transform, 
using the Cauchy-Schwarz inequality together with estimates similar to~\e{n92}, 
we obtain
\begin{align*}
\Vert \partial_x f\Vert_{L^\infty}&\leq \int_\xR |\xi| |\hat f| \dxi \\
&=\int_{|\xi|>\lambda} \kappa(|\xi|)^{-1} |\xi|^{-1} \kappa(|\xi|) |\xi|^{2}|\hat f| \dxi \\
&\quad +\int_{|\xi|\leq \lambda} \kappa(|\xi|)^{-1} (|\xi|+1)^{-\mez} \kappa(|\xi|) |\xi|(1+|\xi|)^{\mez}|\hat f| \dxi \\&\lesssim\left(\int_{|\xi|>\lam} \frac{1}{|\xi|^{2}\kappa^2(|\xi|)} \dxi \right)^{\mez} \Vert \D^{2,\phi}f\Vert_{L^2}\\
&\quad+ \left(\int_{|\xi|\leq \lam} \frac{1}{(|\xi|+1)\kappa^2(|\xi|)} \dxi \right)^{\mez}
\left( \Vert \D^{3/2,\phi}f\Vert_{L^2}+ \Vert f\Vert_{L^2}\right).
\end{align*}
Now observe that
$$
\int_{|\xi|>\lam} \frac{1}{|\xi|^{2}\kappa^2(|\xi|)} \dxi \le 
\frac{1}{\kappa^2(\lam)} \int_{|\xi|>\lam} \frac{1}{|\xi|^{2}} \dxi \leq  \frac{1}{\kappa^2(\lam)\lam}\cdot
$$
It remains to estimate the second integral. 
Remembering that $\kappa(r)\ge \log(4+r)^a$ by assumption, we begin by writing that
$$
\frac{1}{(1+r)\kappa^2(r)}\le \frac{4}{(4+r)\log(4+r)^{2a}}.
$$

On the other hand, with $\beta=(1-2a)/a$ we have $\beta\ge 0$ (since $a< 1/2$) and moreover
$$
\frac{1}{(4+r)\log(4+r)^{2a}}=\frac{1}{a\beta}\fracr \log(4+r)^{a\beta}.
$$
Therefore,
$$
\left(\int_{|\xi|\leq \lam} \frac{1}{(|\xi|+1)\kappa^2(|\xi|)} \dxi\right)^{\frac{1}{2}} \les  \log(4+\lam)^{\frac{1-2a}{2}}.
$$
We conclude that
\begin{align*}
\Vert \partial_x f\Vert_{L^\infty}
&\lesssim \kappa(\lam)^{-1}\lam^{-1/2}\Vert \D^{2,\phi}f\Vert_{L^2}+ \log(4+\lam)^{\frac{1-2a}{2}}\left( \Vert \D^{3/2,\phi}f\Vert_{L^2}+ \Vert  f\Vert_{L^2}\right)
\\ &\lesssim \log(4+\lam)^{\frac{1-2a}{2}}\left(\lam^{-1/2}\Vert \D^{2,\phi}f\Vert_{L^2}+\Vert \D^{3/2,\phi}f\Vert_{L^2}+ \Vert  f\Vert_{L^2}\right).
\end{align*}
Remembering that $\Vert  f\Vert_{L^2}\le \Vert  f_0\Vert_{L^2}$ (see~\e{n51}) and then 
choosing $\lambda$ such that
$$
\lambda^{1/2}=\frac{\Vert \D^{2,\phi}f\Vert_{L^2}}{\Vert \D^{3/2,\phi}f\Vert_{L^2}+ \Vert f_0\Vert_{L^2}},
$$ 
we obtain~\e{linf'}. This completes the proof.
\end{proof}
Now the energy estimate~\e{Z21'} follows directly from Lemma~\ref{L:3.4} and Lemma~\ref{L:3.5}. 
\end{proof}

For later purposes, we conclude this paragraph by recording 
a corollary of the inequalities used to prove Lemma~\ref{L:3.5}.
\begin{corollary}
Consider a function $f\in \mathcal{S}(\xR)$ and set
$$
M=\lA f\rA_{\frac{3}{2},\frac{1}{3}}+\lA f\rA_{L^2}.
$$
Then there holds
\begin{equation}\label{Z103}
\lA f\rA_{\dot H^2}	+\lA \mathcal{T}(f)f\rA_{\dot H^1}
\le C(M+1)^2\log\left(4+\frac{\lA f\rA_{2,\frac{1}{3}}}{M}\right)^{-\frac{1}{6}}
\lA f\rA_{2,\frac{1}{3}},
\end{equation}
for some absolute constant $C$ independent of $M$.
\end{corollary}
\begin{proof}
It follows from the proof of \e{linf'} with $\phi(r)=\log(4+r)^{\frac{1}{3}}$ that we have the two following estimates:
	\begin{align*}
	&	\int|\xi| |\hat f(\xi)|\dxi\lesssim  \log\left(4+\frac{||f||_{2,\frac{1}{3}}}{M}\right)^{\frac{1}{6}}M,\\&
	||f||_{\dot H^2}\lesssim \log\left(4+\frac{||f||_{2,\frac{1}{3}}}{M}\right)^{-1/3} ||f||_{2,\frac{1}{3}}.
	\end{align*}
 Therefore, by combining these with  \eqref{Z101} and \eqref{Z102}, we get that 
	\begin{align*}
	&||f||_{\dot H^2}	+||\mathcal{T}(f)f||_{\dot H^1}\\&\lesssim \left(\lA f\rA_{\dot H^{\frac32}}+\lA f\rA_{\dot H^{\frac32}}^2+1+\log\left(4+\frac{||f||_{2,\frac{1}{3}}}{M}\right)^{\frac{1}{6}}M\right) \log\left(4+\frac{||f||_{2,\frac{1}{3}}}{M}\right)^{-\frac{1}{3}} ||f||_{2,\frac{1}{3}}\\&\lesssim (M+1)^2\log\left(4+\frac{||f||_{2,\frac{1}{3}}}{M}\right)^{-\frac{1}{6}}||f||_{2,\frac{1}{3}},
	\end{align*}
which is the wanted result~\eqref{Z103}. 
\end{proof}

\subsection{Uniform estimates for small initial data globally in time}\label{S:3.3}
In this paragraph, we apply Proposition~\ref{P:3.3} to obtain uniform estimates globally in time, 
assuming some smallness assumption.

\begin{proposition}\label{pro1}
There exists two positive constants $c$ and $C$ such that the following 
property holds. 
For all initial data $f_0$ in $\mathcal{H}^{\tdm,\frac{1}{3}}(\xR)$ satisfying 
\be\label{nUG0}
\lA f_0\rA_{\tdm,\frac13}\left(\lA f_0\rA_{L^2}^2+1\right)  \leq  c,
\ee
where the semi-norm $\lA \cdot\rA_{\tdm,\frac{1}{3}}$ is as defined in~\e{n141}, and 
for all integer $n$ in $\xN\setminus\{0\}$,  the solution 
$f_n$ to the approximate Cauchy problem~\e{A3} satisfies
\begin{equation}\label{nUG1}
\sup_{t\in [0,+\infty)}\lA f_n(t)\rA_{\tdm,\frac13}\leq  \lA f_0\rA_{\tdm,\frac13},
\end{equation}
together with
\begin{equation}\label{nUG2}
\int_0^{+\infty}\bigg[[\lA \partial_tf_n\rA_{\dot{H}^1}^2+\lA f_n\rA_{\dot{H}^2}^2
+\lA\mathcal{T}(f_n)f_n\rA_{\dot H^1}^2+\frac{ \lA f_n\rA_{2,\frac{1}{3}}^2}{\log(4+\lA f_n\rA_{2,\frac{1}{3}})^{\frac{1}{3}}}\bigg]\dt \leq  C\lA f_0\rA_{\tdm,\frac13}^2.
\end{equation}

Furthermore, there exists a subsequence of $(f_n)$ converging to a solution $f$  of the Muskat equation. 
Also,  $f$ satisfies \eqref{nUG1} and \eqref{nUG2} with $f_n=f$. 
\end{proposition}
\begin{proof}
Fix $\kappa(r)=\big(\log(4+r)\big)^{\frac{1}{3}}$, 
define $\phi$ by~\e{n10} and then consider $A_n$ and $B_n$ as given by~\e{n67}. 
Notice that, since $\phi\sim \kappa$, we have
$$
\blA \D^{\tdm,\phi}g\brA_{L^2}\sim\lA g\rA_{\tdm,\frac{1}{3}}.
$$

The estimate~\e{Z21'} implies that
\be\label{Zglobal}
\begin{aligned}
\fract A_n(t)&+C_1\frac{B_n(t)}{1+ \log\left(4+\frac{B_n(t)}{A_n(t)+ \Vert f_0\Vert_{L^2}^2}\right)^{\frac{1}{3}}
\left( A_n(t)+ \Vert f_0\Vert_{L^2}^2\right)}\\
&\leq C_2 \left( \sqrt{A_n(t)}+A_n(t) \right)\log\left(4+\frac{B_n(t)}{A_n(t)}\right)^{-\frac{1}{3}}B_n(t)
\\
&\leq C_2 \left( \sqrt{A_n(t)}+A_n(t) \right)\log\left(4+\frac{B_n(t)}{A_n(t)+ \Vert f_0\Vert_{L^2}^2}\right)^{-\frac{1}{3}}B_n(t).
\end{aligned}
\ee
We want to absorb the right-hand side by the left-hand side. 
To do so, we shall prove that
\begin{multline}\label{Z5}
C_2 \left(  \sqrt{A_n(t)}+A_n(t)  \right)\log\left(4+\frac{B_n(t)}{A_n(t)+\Vert f_0\Vert_{L^2}^2}\right)^{-\frac{1}{3}}  \\
\le \frac{1}{2}\frac{C_1}{1+ \log\left(4+\frac{B_n(t)}{A_n(t)+ \Vert f_0\Vert_{L^2}^2}\right)^{\frac{1}{3}}
\left( A_n(t)+ \Vert f_0\Vert_{L^2}^2\right)}\cdot
\end{multline}
Set
$$
X=\sqrt{A_n(t)}+A_n(t),\quad Y=A_n(t)+\Vert f_0\Vert_{L^2}^2,\quad 
\lambda =\log\left(4+\frac{B_n(t)}{A_n(t)+ \Vert f_0\Vert_{L^2}^2}\right)^{\frac{1}{3}}.
$$
Then \e{Z5} is equivalent to
$$
C_2 X \le \frac{C_1}{2}\frac{\lambda}{1+\lambda Y}.
$$
The latter inequality will be satisfied provided that $2C_2X(Y+1)\le C_1$. 
This means that \e{Z5} will be satisfied provided that
\begin{align}\label{Z6}
C_2 \left( \sqrt{A_n(t)}+A_n(t) \right)\left(A_n(t)+ \Vert f_0\Vert_{L^2}^2+1\right)  \leq \frac{C_1}{2}\cdot
\end{align}
We thus have proved that if \e{Z6} is true for all time $t$, then \e{Z5} is also true for all time. 
On the other hand, let us assume that \e{Z5} is true for all time. Then 
\e{Zglobal} implies that
\be\label{nUG4}
\fract A_n(t)+\frac{C_1}{2}\delta_n(t)B_n(t)\le 0.
\ee
This immediately implies that $A_n$ is decreasing, which implies that \e{Z6} is true also for time $t$  
provided that it holds at initial time. 
By an elementary continuity argument, one can make the previous reasoning rigorous. 
This proves that~\e{nUG1} holds provided that the assumption~\e{nUG0} is satisfied 
with
$$
c=\frac{C_1}{8(C_1+C_2)}\cdot
$$

Integrating~\e{nUG4} in time and noticing that 
$\delta_n(t)\gtrsim \log(4+B_n(t))^{-\frac{1}{3}}$, we also get that the function $t\mapsto  B_n(t)\log(4+B_n(t))^{-\frac{1}{3}}$ is integrable on $[0,+\infty)$. By virtue 
of~\eqref{Z103} and using the equation $\partial_t f_n=-\D f_n+J_n\big(\mathcal{T}(f_n)f_n\big)$, 
we end up with~\e{nUG2}. 
Eventually, by the standard compactness theorem, there exists a subsequence of $(f_n)$ converging to 
a solution $f$ 
of the Muskat equation. 
\end{proof}

\subsection{Uniform estimates for arbitrary initial data}\label{S:critical}
We now prove uniform estimates for arbitrary initial data in $\mathcal{H}^{\tdm,\frac{1}{3}}(\xR)$, without any smallness assumption. This is the most delicate step. Indeed, 
as explained in Remark~\ref{R:1.4}, one important feature of this 
problem is that 
the estimates will not only depend on the norm of the initial data: 
they depend on the initial data themselves. As a consequence, 
we are forced to estimate 
the approximate solutions $f_n$ for a norm whose definition 
depends on the initial data. 
More precisely, we will estimate 
the norm $\D^{\tdm,\phi}f_n$ for some function $\phi$ depending on $f_0$. 
To define this function $\phi$, we begin with the following general lemma.

\begin{lemma}\label{L:critical}
For any nonnegative integrable function 
$\omega\in L^1(\mathbb{R})$, there exists a function  $\eta\colon[0,\infty) \to [1,\infty)$ satisfying the following properties: 
\begin{enumerate}
\item $\eta$ is increasing and $\lim\limits_{r\to \infty}\eta(r)=\infty$,
\item $\eta(2r)\leq 2\eta(r)$ for any $r\geq 0$,
\item $\omega$ satisfies the enhanced integrability condition:
\begin{equation}
\int_\xR \eta(|r|) \omega(r) \dr<\infty,
\end{equation}
\item moreover, the function $r\mapsto \eta(r)/\log(4+r)$ is decreasing on  $[0,\infty)$.
\end{enumerate}
\end{lemma}
\begin{proof}
Consider a sequence of real-number $(\alpha_k)_{k\ge 1}$ such that 
$\alpha_{1}\geq e^{5}$ and $\alpha_{k}\geq  \alpha_{k-1}^{10}$ and in addition 
\begin{equation}
\forall k\ge 1,\qquad \int_{|r|\geq \alpha_k}\omega (r)\dr\leq 2^{-k}.
\end{equation}
We set
\begin{equation}\label{ODE1}
\eta(r)=\left\{ 
\begin{aligned}
&2 ~~~ &\text{if }&~~0\leq r< \alpha_1,\\ 
&k+1+\frac{\log(\frac{4+r}{4+\alpha_{k}})}{\log(\frac{4+\alpha_{k+1}}{4+\alpha_{k}})} \qquad&\text{if }&~~\alpha_{k}\leq r< \alpha_{k+1}.
\end{aligned} \right.
\end{equation}
It is easy to check that $\eta\colon [0,\infty) \to [1,\infty)$ is an increasing function 
converging to $+\infty$ when $r$ goes to $+\infty$. 
Moreover, $\eta$ satisfies $\eta(2r)\leq 2\eta(r)$ for any $r\geq 0$. 

In addition,
\begin{align*}
\int \eta(|r|) \omega(r)\dr&\leq \int_{|r|\leq \alpha_{1}} 2 \omega(r)\dr+\sum_{k=1}^{\infty}(k+2)\int_{\alpha_{k}\leq |r|\leq \alpha_{k+1}}\omega(r)\dr\\
&\leq   2 ||\omega||_{L^1}+\sum_{k=1}^{\infty}(k+2)2^{-k}\\&\leq   2 ||\omega||_{L^1}+C.
\end{align*}
It remains to prove that $r\mapsto \eta(r)/\log(4+r)$ is decreasing. To do so, write
\begin{equation}
\fracr\left(\frac{\eta(r)}{\log(4+r)} \right)=\frac{1}{\log(4+r)} \left(\eta'(r)-\frac{1}{4+r}\frac{\eta(r)}{\log(4+r)}\right).
\end{equation}
So, for $0\leq r< \alpha_1$,
\begin{equation}
\fracr\left(\frac{\eta(r)}{\log(4+r)} \right)<0,
\end{equation}
while for $\alpha_k\leq r<\alpha_{k+1}$ with $k\ge 1$, we have
\begin{align*}
\fracr\left(\frac{\eta(r)}{\log(4+r)} \right)&\leq \frac{1}{(4+r)\log(4+r)^2} \left(\frac{\log(4+r)}{\log(\frac{4+\alpha_{k+1}}{4+\alpha_{k}})}-k-1\right)\\&\leq  \frac{1}{(4+r)\log(4+r)^2} \left(\frac{\log(4+\alpha_{k+1})}{\log(\frac{4+\alpha_{k+1}}{4+\alpha_{k+1}^{1/10}})}-2\right)<0,
\end{align*}
where we have used  $\alpha_{k+1}\geq e^{5\times 10^k}$.

This proves that $r\mapsto \eta(r)/\log(4+r)$ is decreasing on $[0,\infty)$. The proof is complete.
\end{proof}
After this short d\'etour, we return to the main line of our development. 
Consider a function $f_0$ in $\mathcal{H}^{\tdm,\frac{1}{3}}(\xR)$. 
It immediately follows from the previous lemma and Plancherel's theorem that there exists an 
function $\tilde{k}\colon[0,\infty) \to [1,\infty)$ 
such that
\begin{equation}\label{n131}
\int_\xR |\xi|^3 \log\big(4+|\xi|^2\big)^{\frac{2}{3}}(\tilde{k}(\xi))^2 \bla \hat f_0(\xi)\bra^2 \dxi <+\infty,
\end{equation}
and such that $\tilde{k}$ is increasing, 
$r\mapsto \tilde{k}(r)/\log(e+r)$ is decreasing; 
$\tilde{k}(2r)\leq c_0\tilde{k}(r)$ and $\lim\limits_{r\to \infty}\tilde{k}(r)=\infty$.

Next we now define a function $\kappa_0\colon [0,+\infty)\to[1,+\infty)$ by
$$
\kappa_0(r)=\big(\log(4+|\xi|)\big)^{\frac{1}{3}} \, \tilde{k}(|\xi|).
$$
together with the companion function $\phi_0$ defined by~\e{n10}, that is
\begin{equation}\label{defi:phi0}
\phi_0(\lam)=\int_{0}^{\infty}\frac{1-\cos(h)}{h^2} \kappa_0\left(\frac{\lam}{h}\right) \dh, \quad \text{for }\lambda\ge 0.
\end{equation}
\begin{proposition}
Consider an initial data $f_0$ in $\mathcal{H}^{\tdm,\frac{1}{3}}(\xR)$ and denote by 
$\phi_0$ the function defined above in~\e{defi:phi0}. Set 
\begin{equation*}
M_0=\blA \D^{\tdm,\phi_0}f_0\brA_{L^2}^2.
\end{equation*} Then, there exists $T_0>0$ depending on $M_0$ and   $\Vert f_0\Vert_{L^2}$ such that, 
for any integer $n$ in $\xN\setminus\{0\}$, the solution 
$f_n$ to the approximate Cauchy problem~\e{A3} satisfies
\begin{equation}
\sup_{t\in [0,T_0]}\blA \D^{\tdm,\phi_0}f_n(t)\brA_{L^2}^2\leq 2M_0
\end{equation}
and 
\begin{equation}\label{Z104}
\int_0^{T_0}\left(\lA \partial_tf_n\rA_{\dot{H}^1}^2+\lA f_n\rA_{\dot{H}^2}^2+||\mathcal{T}(f_n)(f_n)||_{\dot H^1}^2+\frac{ ||f_n||_{2,\frac{1}{3}}^2}{\log(4+||f_n||_{2,\frac{1}{3}})^{\frac{1}{3}}}\right)\dt \leq  CM_0,
\end{equation}
for some absolute constant  $C>0$ independent of $f_0$. 

Furthermore, there exists a subsequence of $(f_n)$ converging to a solution $f$  of the Muskat equation 
which satisfies~\eqref{nUG1} and \eqref{nUG2} with $f_n$ replaced by~$f$.

\end{proposition}
\begin{remark}\label{R:3.8}
Notice that the time $T_0$ depends on $f_0$ and not only on~$\blA f_0\brA_{\mathcal{H}^{\tdm,\frac{1}{3}}}$. 
\end{remark}
\begin{proof}
We apply \eqref{Z21'} for the quantities
\begin{align*}
A_n(t)=\blA \D^{\tdm,\phi_0}f_n(t)\brA_{L^2}^2, \qquad B_n(t)=\blA \D^{2,\phi_0}f_n(t)\brA_{L^2}^2.
\end{align*}
This gives that
\be\label{Z21''}
\fract A_n(t)+C_1\delta_n(t)B_n(t)\leq C_2 \left( \sqrt{A_n(t)}+A_n(t) \right)\mu_n(t) B_n(t),
\ee
where
\begin{align*}
\delta_n(t)&=\left(1+ \left[\log\left(4+\frac{B_n(t)}{A_n(t)+ \Vert f_0\Vert_{L^2}^2}\right)\right]^{1-2a}\left( A_n(t)+ \Vert f_0\Vert_{L^2}^2\right)\right)^{-1},\\
\mu_n(t)&=\left(\log\left(4+\frac{B_n(t)}{A_n(t)}\right)\right)^{-\frac{1}{3}}\times \left(\tilde{k}
\left(\frac{B_n(t)}{A_n(t)}\right)\right)^{-1}.
\end{align*}
Given $\varrho\ge 0$, define the function
\begin{align*}
\mathcal{E}\left(\varrho,\Vert f_0\Vert_{L^2}^2\right)
=\sup_{r\ge 0} \Bigg\{&C_2 
\frac{\left(\sqrt{\varrho}+\varrho\right)r}{\tilde{k}\left(\frac{r}{\varrho}\right)\left[\log\left(4+\frac{r}{\varrho}\right)\right]^{1/3}}
\\
&-\frac{C_1}{2}\frac{r}{1+ \Big[\log\Big(4+\frac{r}{\varrho}\Big)\Big]^{1/3}
\left(\varrho+ \Vert f_0\Vert_{L^2}^2\right)}\Bigg\}.
\end{align*}
Since $\rho\mapsto\tilde{k}(\rho)$ is increasing, 
directly from the definition of $\mathcal{E}\left(\varrho,\Vert f_0\Vert_{L^2}^2\right)$, we verify that 
the function $\varrho\to \mathcal{E}(\varrho,\Vert f_0\Vert_{L^2}^2)$ is increasing. 
On the other hand, since $\kappa(\rho)$ tends to $+\infty$ as $\rho$ goes to $+\infty$, we verify that 
$$
\forall \varrho\ge 0,\qquad \mathcal{E}(\varrho,\Vert f_0\Vert_{L^2}^2)<\infty.
$$
Thus, 
\begin{equation}
\fract A_n(t)+\frac{C_1}{2}\delta_n(t)B_n(t)\leq \mathcal{E}\left(A_n(t),\Vert f_0\Vert_{L^2}^2\right).
\end{equation}
and {\em a fortiori}
\begin{equation*}
\fract A_n(t)\leq \mathcal{E}\left(A_n(t),\Vert f_0\Vert_{L^2}^2\right).
\end{equation*}
Then by standard arguments, one obtains
\begin{equation}
\sup_{t\in [0,T_0]}\blA \D^{\tdm,\phi_0}f_n(t)\brA_{L^2}^2\leq 2M_0.
\end{equation}
with \begin{equation*}
T_0=\frac{M_0}{\mathcal{E}\left(2M_0,\Vert f_0\Vert_{L^2}^2\right)},~~~M_0=\blA \D^{\tdm,\phi_0}f_0\brA_{L^2}^2.
\end{equation*}Moreover, as proof of Propostion \ref{pro1}, we also have \eqref{Z104} and   a subsequence of $(f_n)$ converging to a solution $f$  of the Muskat equation. 
Also,   $f$ satisfies \eqref{nUG1} and \eqref{nUG2} with $f_n=f$. 
The proof is complete. 
\end{proof}
\subsection{Uniqueness}\label{S:3.5}
The following proposition implies that the solution of the Muskat equation is unique. 
\begin{proposition}
Consider two solutions $f_1,f_2$ of the Muskat equation in $[0,T]\times\mathbb{R}$ (for some $T<\infty$), 
with initial data $f_{1,0},f_{2,0}$ respectively, satisfying 
\begin{equation}\label{Z105}
\sup_{t\in [0,T]}\lA f_k(t)\rA_{\tdm,\frac13}^2+	\int_0^{T} \log\Big(4+\lA f_k\rA_{2,\frac{1}{3}}\Big)^{-\frac{1}{3}} ||f_k||_{2,\frac{1}{3}}^2\dt \leq M<\infty,~~k=1,2.
\end{equation}
Then the difference $g=f_1-f_2$ is estimated by
	\begin{equation}
	\label{Z107} \sup_{t\in [0,T]} \Vert g(t) \Vert_{\dot{H}^\mez}\leq   \Vert g(0) \Vert_{\dot{H}^\mez}\exp\left(C(M)\sum_{k=1}^2\int_0^T\log\left(4+||f_k||_{2,\frac{1}{3}}\right)^{-\frac{1}{3}}
	||f_k||_{2,\frac{1}{3}}^2 \dt \right).
	\end{equation}
\end{proposition}
\begin{proof}
Since $\partial_tf_k+\D f_k = \mathcal{T}(f_k)f_k$, it follows from the decomposition~\eqref{n1} 
of $\mathcal{T}(f_k)f_k$ that the difference $g=f_1-f_2$ satisfies
\begin{align*}
\partial_tg+\frac{\D g}{1+(\partial_xf_1)^2}&=  V(f_1)\partial_x g+R(f_1,g)+\left(\mathcal{T}(f_2+g)-\mathcal{T}(f_2)\right)f_2.
\end{align*}
Take the $L^2$-scalar product of this equation with $\D g$ to get
\begin{align*}
\frac{1}{2}  \fract\Vert g \Vert^{2}_{\dot{H}^\mez}+\int  \frac{( \D g)^2}{1+(\partial_x f_{1})^2} \dx
&\leq \left|\big(V(f_1)\partial_x g,|D| g\big)\right|+\lA R(f_1,g)\rA_{L^2}\lA g\rA_{\dot H^1}\\
&\quad+\lA \left(\mathcal{T}(f_2+g)-\mathcal{T}(f_2)\right)f_2\rA_{L^2}\lA g\rA_{\dot H^1}.
\end{align*}
The arguments used to show the identity \eqref{Z106} also implies that
$$
\la \big(V(f_1)\partial_x  g,\D g\big)\ra
=\mez\la \big( \big[ \mathcal{H},V(f_1)\big]\D g,\D g\big)\ra.
$$
Thus, by combining the estimate \eqref{Z3} for $\lA V(f)\rA_{\dot{H}^1}$ together with 
the commutator estimate~\eqref{Z2} about the Hilbert transform, 
the bound~\eqref{Z14} for the 
remainder term and eventually the following interpolation inequality (see \e{linf'}),
$$
\lA \partial_xf_{1}\rA_{L^\infty}\le C(M)\log\left(4+\lA f_1\rA_{2,\frac{1}{3}}\right)^{\frac{1}{6}},
$$
we end up with
\begin{align*}
 \fract\Vert g \Vert^{2}_{\dot{H}^\mez}+&C(M)\log\Big(4+||f_1||_{2,\frac{1}{3}}\Big)^{-\frac{1}{3}} ||g||_{\dot H^1}^2\\&~~~\lesssim  \left(\lA f_1\rA_{\dot{H}^2}+\lA f_1\rA_{\dot{H}^{\frac{7}{4}}}^2\right)\lA g\rA_{\dot H^{\mez}}\lA g\rA_{\dot H^{1}}+\lA f_1\rA_{\dot{H}^{\frac{7}{4}}}\Vert g\Vert _{\dot{H}^{\frac{3}{4}}}|| g||_{\dot H^1}\\
 &\quad+\lA \left(\mathcal{T}(f_2+g)-\mathcal{T}(f_2)\right)f_2\rA_{L^2}\lA g\rA_{\dot H^1}.
 \end{align*}
Now, directly from the definition of $\mathcal{T}(f)g$, we have
$$
|\left(\mathcal{T}(f_2+g)-\mathcal{T}(f_2)\right)f_2(x)|\lesssim \int |\Delta_{\alpha} (\partial_xf_{2})(x)||\Delta_{\alpha} g(x)| \dalpha,
$$
so, by H\"older's inequality, the $L^2$-norm of $(\mathcal{T}(f_2+g)-\mathcal{T}(f_2))f_2$ is estimated by
\begin{align*}
\left(\int\left(\int |\Delta_{\alpha}(\partial_xf_{2})(x)||\Delta_{\alpha} g(x)| \dalpha\right)^{2}\dx\right)^{\mez}
\leq \lA (\partial_xf_{2})\rA_{\dot{F}^{\frac{1}{2}}_{4,2}}\lA g\rA_{\dot{F}^{\frac{1}{2}}_{4,2}}\lesssim
\lA f_2\rA_{\dot{H}^{\frac{7}{4}}}\Vert g\Vert _{\dot{H}^{\frac{3}{4}}}.
\end{align*}
Hence, we derive 
\begin{align*}
\fract\Vert g \Vert^{2}_{\dot{H}^\mez}+&C(M)\log\Big(4+||f_1||_{2,\frac{1}{3}}\Big)^{-\frac{1}{3}} ||g||_{\dot H^1}^2\\&~~~\lesssim  \left(\lA f_1\rA_{\dot{H}^2}+\lA f_1\rA_{\dot{H}^{\frac{7}{4}}}^2\right)\lA g\rA_{\dot H^{\mez}}\lA g\rA_{\dot H^{1}}+\lA f_2\rA_{\dot{H}^{\frac{7}{4}}}\Vert g\Vert _{\dot{H}^{\frac{3}{4}}}|| g||_{\dot H^1}
\\&~~~\lesssim  \left(\lA f_1\rA_{\dot{H}^2}+\lA f_1\rA_{\dot{H}^{\frac{7}{4}}}^2\right)\lA g\rA_{\dot H^{\mez}}\lA g\rA_{\dot H^{1}}+\lA f_2\rA_{\dot{H}^{\frac{7}{4}}}\Vert g\Vert _{\dot{H}^{\frac{1}{2}}}^{\frac{1}{2}}|| g||_{\dot H^1}^{\frac{3}{2}}.
\end{align*}
Interchanging the role of $f_1$ and $f_2$, we also get a symmetric estimate. Then, by combining these two estimates, we get
\begin{align*}
\fract\Vert g \Vert^{2}_{\dot{H}^\mez}+&C(M)\bigg[
\log\Big(4+||f_1||_{2,\frac{1}{3}}\Big)^{-\frac{1}{3}}+\log\Big(4+||f_2||_{2,\frac{1}{3}}\Big)^{-\frac{1}{3}}\bigg] ||g||_{\dot H^1}^2\\
&~~~\lesssim  \left(\lA f_1\rA_{\dot{H}^2}+\lA f_1\rA_{\dot{H}^{\frac{7}{4}}}^2\right)\lA g\rA_{\dot H^{\mez}}\lA g\rA_{\dot H^{1}}+\lA f_2\rA_{\dot{H}^{\frac{7}{4}}}\Vert g\Vert _{\dot{H}^{\frac{1}{2}}}^{\frac{1}{2}}|| g||_{\dot H^1}^{\frac{3}{2}}\\
&~~~\quad +\left(\lA f_2\rA_{\dot{H}^2}+\lA f_2\rA_{\dot{H}^{\frac{7}{4}}}^2\right)\lA g\rA_{\dot H^{\mez}}\lA g\rA_{\dot H^{1}}+\lA f_1\rA_{\dot{H}^{\frac{7}{4}}}\Vert g\Vert _{\dot{H}^{\frac{1}{2}}}^{\frac{1}{2}}|| g||_{\dot H^1}^{\frac{3}{2}}.
\end{align*}

By  interpolation  inequality \eqref{Z20'}
\begin{align*}
&\lA f_k\rA_{\dot H^{\frac{7}{4}}}\lesssim_M \log\left(4+\lA f_k\rA_{2,\frac{1}{3}}\right)^{-\frac{1}{3}}
\lA f_k\rA_{2,\frac{1}{3}}^{\mez},\\
&\lA f_k\rA_{\dot H^2}\lesssim_M \log\left(4+\lA f_k\rA_{2,\frac{1}{3}}\right)^{-\frac{1}{3}}
\lA f_k\rA_{2,\frac{1}{3}},
\end{align*}
hence 
\begin{align*}
&\fract\Vert g \Vert^{2}_{\dot{H}^\mez}+C(M)\bigg[
\log\Big(4+||f_1||_{2,\frac{1}{3}}\Big)^{-\frac{1}{3}}+\log\Big(4+||f_2||_{2,\frac{1}{3}}\Big)^{-\frac{1}{3}}\bigg] \lA g\rA_{\dot H^1}^2\\
&\qquad\qquad\lesssim_M \log\left(4+\lA f_1\rA_{2,\frac{1}{3}}\right)^{-\frac{1}{3}}
\lA f_1\rA_{2,\frac{1}{3}}\lA g\rA_{\dot H^{\mez}}\lA g\rA_{\dot H^{1}}\\
&\qquad\qquad\quad+\log\left(4+\lA f_2\rA_{2,\frac{1}{3}}\right)^{-\frac{1}{3}}
\lA f_2\rA_{2,\frac{1}{3}}^{\frac{1}{2}}\lA g\rA _{\dot{H}^{\frac{1}{2}}}^{\frac{1}{2}}\lA g\rA_{\dot H^1}^{\frac{3}{2}}\\
&\qquad\qquad\quad+\log\left(4+\lA f_2\rA_{2,\frac{1}{3}}\right)^{-\frac{1}{3}}
\lA f_2\rA_{2,\frac{1}{3}}\lA g\rA_{\dot H^{\mez}}\lA g\rA_{\dot H^{1}}\\
&\qquad\qquad\quad+\log\left(4+\lA f_1\rA_{2,\frac{1}{3}}\right)^{-\frac{1}{3}}
\lA f_1\rA_{2,\frac{1}{3}}^{\frac{1}{2}}\lA g\rA _{\dot{H}^{\frac{1}{2}}}^{\frac{1}{2}}\lA g\rA_{\dot H^1}^{\frac{3}{2}}.
\end{align*}
Finally, by Holder's inequality,
\begin{align*}
\fract\Vert g \Vert^{2}_{\dot{H}^\mez}\lesssim_M 
\bigg(\sum_{k=1}^2 \log\left(4+||f_k||_{2,\frac{1}{3}}\right)^{-\frac{1}{3}}
\lA f_k\rA_{2,\frac{1}{3}}^2\bigg)\lA g\rA_{\dot H^{\mez}}^2,
\end{align*}
which in turn implies \eqref{Z107}. The proof is complete. 
\end{proof}

\section*{Acknowledgments} 
\noindent  Thomas Alazard acknowledges the support of the SingFlows project, grant ANR-18-CE40-0027 
of the French National Research Agency (ANR).  Quoc.-Hung Nguyen is  supported  by the Shanghai Tech University startup fund.

\vfill
\begin{flushleft}
\textbf{Thomas Alazard}\\
Universit{\'e} Paris-Saclay, ENS Paris-Saclay, CNRS,\\
Centre Borelli UMR9010, avenue des Sciences, 
F-91190 Gif-sur-Yvette\\
France.

\vspace{1cm}

\textbf{Quoc-Hung Nguyen}\\
ShanghaiTech University, \\
393 Middle Huaxia Road, Pudong,\\
Shanghai, 201210,\\
China

\end{flushleft}

\end{document}